\documentclass[12pt]{amsart}

\usepackage{amsmath,color}
\usepackage{amsthm}
\usepackage{hyperref}
\usepackage[margin=1.0in]{geometry}

\numberwithin{equation}{section}

\newtheorem{prop}{Proposition}
\newtheorem{lemma}[prop]{Lemma}

\newtheorem{thm}[prop]{Theorem}
\newtheorem{cor}[prop]{Corollary}

\numberwithin{prop}{section}

\theoremstyle{definition}
\newtheorem{defn}[prop]{Definition}

\newtheorem{rmk}[prop]{Remark}

\newcommand{\be}{\beta}
\newcommand{\del}{\partial}
\newcommand{\delb}{\bar{\partial}}
\newcommand{\brs}[1]{\left| #1 \right|}

\newcommand{\gD}{\Delta}

\newcommand{\gs}{\sigma}

\newcommand{\gl}{\lambda}

\newcommand{\gw}{\omega}

\newcommand{\N}{\nabla}

\newcommand{\JJ}{\mathbb J}

\newcommand{\til}[1]{\widetilde{#1}}

\renewcommand{\bar}[1]{\overline{#1}}

\renewcommand{\t}{\mathfrak{t}}

\newcommand{\T}{\mathbb{T}}
\newcommand{\Pol}{P}

\newcommand{\standard}{Delzant}

\newcommand{\IP}[1]{\left<#1\right>}

\newcommand{\Hess}{\mathrm{Hess}}

\newcommand{\Gscal}{R}

\DeclareMathOperator{\Rc}{Rc}

\DeclareMathOperator{\tr}{tr}

\DeclareMathOperator{\divg}{div}

\DeclareMathOperator{\spn}{span}

\newcommand{\tor}{\mathfrak{t}}
\newcommand{\la}{\langle}
\newcommand{\ra}{\rangle}
\newcommand{\C}{\mathbb C}
\newcommand{\R}{\mathbb R}
\newcommand{\Z}{\mathbb Z}

\begin{document}

\title[Toric geometry of Generalized K\"ahler-Ricci solitons]{Toric geometry of Generalized K\"ahler-Ricci solitons}

\author{Vestislav Apostolov}
\address{V.\,Apostolov\\ D{\'e}partement de Math{\'e}matiques\\ UQAM \\
 and \\ Institute of Mathematics and Informatics\\ Bulgarian Academy of Sciences}
\email{\href{mailto:apostolov.vestislav@uqam.ca}{apostolov.vestislav@uqam.ca}}

\author{Giuseppe Barbaro}
\address{G.\,Barbaro\\ Department of Mathematics, Aarhus University, Ny Munkegade 118, 8000 Aarhus C,
Denmark}
\email{\href{mailto:g.barbaro@math.au.dk}{g.barbaro@math.au.dk}}

\author{Jeffrey Streets}
\address{Rowland Hall\\
         University of California, Irvine\\
         Irvine, CA 92617}
\email{\href{mailto:jstreets@uci.edu}{jstreets@uci.edu}}

\author{Yury Ustinovskiy}
\address{Y.\,Ustinovskiy, Susquehanna International Group}
\email{\href{mailto:yura.ust@gmail.com}{yura.ust@gmail.com}}

\date{September 1st, 2025}

\begin{abstract}
    We establish a local equivalence between toric steady K\"ahler-Ricci solitons and $A$-type toric generalized K\"ahler-Ricci solitons (GKRS).  Under natural global conditions we show this equivalence extends to complete GKRS, yielding a general construction of new examples in all dimensions.  We show that in four dimensions, all GKRS are either described by the generalized K\"ahler Gibbons-Hawking ansatz, or have split tangent bundle, or are $A$-type toric. This yields a local classification in four dimensions, together with a conjecturally exhaustive construction of complete symplectic-type examples. 
\end{abstract}

\thanks{V.A. was supported in part by an NSERC Discovery Grant. G.B. is a member of GNSAGA of INdAM and has been supported by Sapere Aude: DFF-Starting Grant ``Conformal geometry: metrics and cohomology.'' J.S. was supported in part by the NSF via DMS-2203536. The first named author thanks O. Biquard, C. Cifarelli, R. Conlon, and A. Lahdili for useful discussions.}

\maketitle

\section{Introduction}

Generalized K\"ahler-Ricci solitons (GKRS) are a natural generalization of Calabi-Yau metrics  beyond the K\"ahler setting.  They arise naturally from diverse considerations in mathematical physics \cite{GHR,ZabzineGCY}, generalized complex geometry \cite{HitchinGCY, GualtieriGCG}, and geometric/renormalization group flows \cite{PCF, GKRF}.  A triple $(g, I, J)$ of a Riemannian metric with compatible integrable complex structures $I$ and $J$ defines a generalized K\"ahler structure if \cite{GHR, GualtieriGKG}
\begin{align*}
    d^c_I \gw_I =&\ H = - d^c_J \gw_J, \qquad d H = 0.
\end{align*}
The data furthermore defines a GKRS if there exists a smooth function $f$ such that $(g, H, f)$ is a generalized Ricci soliton, i.e.
\begin{align*}
    \Rc^{\N^+} +\ \N^+ \N^+ f = 0, \qquad \N^+ = \N + \tfrac{1}{2} g^{-1} H.
\end{align*}
See \cite{GRFbook} for more background on generalized Ricci solitons.  Note that GKRS come equipped with two canonical holomorphic Killing fields (cf. Proposition \ref{p:GKRSsymmetries})
\begin{align*}
    X_I =&\ \tfrac{1}{2} I \left( \theta_I^{\sharp} - \N f \right), \qquad X_J = \tfrac{1}{2} J \left( \theta_J^{\sharp} - \N f \right).
\end{align*}
These define a notion of rank for GKRS which is the maximal dimension of the space spanned by these vector fields over points in the manifold (cf. Definition \ref{d:rank}).  In the case of real dimension $4$ and rank $1$, the second and fourth-named authors gave a complete local classification using a generalization of the Gibbons-Hawking ansatz \cite{SU2}.  Furthermore, they constructed an exhaustive list of complete, simply connected examples and characterized them as the only complete regular examples of rank 1.

Our work is initially inspired by the rank $2$ case in real dimension $4$, where the symmetries above consist of a local toric action.  We will further restrict to the setting of GK structures of symplectic-type, which means that one of the $2$-forms $F_{\pm} = 2 (I \pm J)^{-1} g$ is defined globally.  In dimension four this is always the local behavior on a dense open subset, outside the special case when $I$ and $J$ commute and the tangent bundle admits a global splitting~\cite{ApostolovGualtieri}.  The study of toric GK structures of symplectic-type was initiated by Boulanger \cite{boulanger2019toric}, and furthered by Wang \cite{wang2022toric, wang2023toric, wang2020toric}.  Boulanger gave a construction by deforming the complex structure of a toric K\"ahler metric by a constant matrix in the angular/momentum variables, whereas Wang extended this by also deforming the symplectic form by a constant matrix.  In \cite{ASU2} the authors showed that these corresponded geometrically to a deformation of the underlying Hitchin Poisson tensor \cite{HitchinPoisson}, and referred to these as type A and type B deformations respectively, a terminology we adopt here.  Note that it was shown in (\cite{ASU3} Proposition 4.2) that GKRS of symplectic-type on compact manifolds are necessarily K\"ahler Calabi-Yau, hence our study will focus on the noncompact case.

\subsection{Toric steady K\"ahler-Ricci solitons and generalized K\"ahler-Ricci solitons}

While our overall focus is in four dimensions, given the canonical link to toric geometry mentioned above we will broaden our scope to address toric GKRS in all dimensions.  Our first main result is a local equivalence between symplectic-type toric GKRS and A deformations of a toric steady gradient K\"ahler-Ricci soliton.  We briefly recall that a K\"ahler structure $(g_K, J_K)$ is a steady gradient K\"ahler-Ricci soliton if there exists a smooth function $f_K$ such that $(g_K, f_K)$ defines a steady Ricci soliton, i.e.
\begin{align}\label{d:KRS}
    \Rc_{g_K} +\ \N^2 f_K =&\ 0.
\end{align}
We shall abbreviate such structures as KRS, noting that we only consider \emph{steady} \emph{gradient} K\"ahler-Ricci solitons in this article.  These metrics come equipped with the canonical holomorphic Killing field $-\frac{1}{2} J \N f_K$.  For further background and existence results see \cite{apostolov2023hamiltonian, Biquard-Macbeth, bryant-KRS, cao1996existence, Cifarelli, Conlon-Deruelle, Dancer-Wang, Futaki-Wang, Schafer}.

\begin{thm} \label{t:Atypeconstruction} Let $(\mathring{M}, F, g, I, J, f)$ be a $2n$-dimensional symplectic-type GKRS admitting Killing fields $\{ X_1, \ldots, X_n \}$ of $(F, g, I, J)$ such that at each point of $\mathring{M}$ one has
\begin{enumerate}
    \item $\tor:={\rm span}_{\mathbb R} \{X_1, \ldots X_n\}$ is $n$-dimensional, 
    \item the canonical soliton vector fields $X_I, X_J \in \tor$,
    \item $F_{|_{\tor \times \tor}}\equiv 0$,
    \item $I \tor = J \tor$.
\end{enumerate} 
Then there exists an $F$-compatible, $(X_1, \ldots, X_n)$-invariant gradient steady K\"ahler-Ricci soliton $(g_K, J_K, f_K)$ on $\mathring{M}$, with soliton vector field $-\frac{1}{2} J_K \N f_K = \tfrac{1}{2}\left(X_I + X_J\right)$.  Furthermore there exists $A \in \Lambda^2 \tor$ such that $(F, g, I, J)$ is the $A$-deformation of $(F, g_K, J_K)$.

Conversely, suppose $(\mathring{M}, F, g_K, J_K, f_K)$ is a $2n$-dimensional gradient steady K\"ahler-Ricci soliton admitting Killing fields $\{X_1,\dots,X_n\}$ of $(F, g_K, J_K)$ such that at each point of $\mathring{M}$ one has
\begin{enumerate}
    \item $\tor:={\rm span}_{\mathbb R}\{X_1, \ldots X_n)$ is $n$-dimensional,
    \item the canonical soliton vector field $- \frac{1}{2} J \N f_K \in \tor$,
    \item $F_{|_{\tor \times \tor}}\equiv 0$.
\end{enumerate}
Then given $A \in \Lambda^2 \tor$, there exists a smooth function $f$ such that the $A$-deformation of $(g_K, J_K)$ is a symplectic-type GKRS satisfying the conditions above.
\end{thm}

\noindent For the proof of the first part we note that it follows from the prior work on toric generalized K\"ahler structures \cite{boulanger2019toric} that conditions (3) and (4) imply the structure is locally an $A$-deformation of a K\"ahler structure.  Further delicate local computations show that this is in fact a gradient K\"ahler-Ricci soliton.  These computations can essentially be reversed to give the proof of the second part, after deriving the relevant soliton potential in Lemma \ref{GKRS}.  Building upon the relationship between KRS and GKRS, we establish a duality result for toric K\"ahler-Ricci solitons stemming from Legendre duality in \S \ref{s:Legendre}.

The above result can be applied under some global assumptions in order to relate complete GKRS with complete KRS.  In particular we assume that $(M^{2n}, F)$ is a non-compact toric symplectic manifold, obtained from the Delzant construction associated to a unbounded simple convex Delzant polytope $\Pol \subset \R^n$, see \cite{karshon2015non, abreu2012scalar, cifarelli2022uniqueness} for the theory of such manifolds. In this situation, $(M, F)$ admits a Hamiltonian action of an $n$-dimensional compact torus $\T$ whose momentum map $\mu : M \to \R^n$ is proper and has image $\Pol$;  moreover, if  $\mathring{\Pol}$ denotes the interior of $\Pol$, then the open dense subset $\mathring{M}:= \mu^{-1}(\mathring{\Pol})$ of $M$ satisfies the first two conditions in Theorem~\ref{t:Atypeconstruction}.  We first  show in \S \ref{ss:extension} that a $\T$-invariant GKRS $(g, I, J)$ on $\mathring{M}$ is globally defined on $M$ if and only if the corresponding gradient steady KRS (which is defined on $\mathring{M}$ via Theorem~\ref{t:Atypeconstruction}) is globally defined on $M$.  This is an extension of results in the compact case \cite{boulanger2019toric, wang2020toric}, established here in the noncompact setting using ideas of \cite{apostolov2004hamiltonian,  sena2021uniqueness}.  Furthermore, in this setting we show equivalence of completeness of the two solitons:
\begin{thm} \label{t:completeness} (cf. Theorem \ref{t:completeness_text}) Suppose $(M, F, \T)$ is a non-compact toric symplectic manifold obtained by the Delzant construction with respect to a simple convex unbounded polytope $\Pol \subset \R^n$. Let $(g, I, J,f)$ and $(g_K, J_K, f_K)$ be respectively a $\T$-invariant GKRS and $\T$-invariant KRS on 
$M$ related by an $A$-deformation over $\mathring{M}$ as in Theorem~\ref{t:Atypeconstruction}.  Then $g$ is complete if and only if $g_K$ is complete.
\end{thm}

\noindent In fact we will show that completeness of the GKRS implies completeness of the associated KRS even without the Delzant condition (cf. Proposition \ref{p:complete}).  The reverse implication requires a combination of arguments from the theory of length metric spaces together with a precise understanding of the Riemannian geometry of the manifold-with-corners induced by the moment map.  These arguments are perhaps known to experts but we include them as we did not find a precise reference.

Theorems \ref{t:Atypeconstruction} and \ref{t:completeness} yield large classes of complete symplectic-type GKRS in all dimensions from constructions of complete toric KRS in \cite{apostolov2023hamiltonian,Biquard-Macbeth, cao1996existence, Cifarelli, Conlon-Chau-Lai, Conlon-Deruelle, Futaki-Wang, Schafer}.  For instance, the general existence and uniqueness results in \cite{Biquard-Macbeth, Conlon-Deruelle} for complete steady KRS's and Theorem~\ref{t:completeness} yield:
\begin{cor}\label{c:type-A}  Any equivariant toric crepant resolution of a toric Calabi-Yau K\"ahler cone admits  complete  generalized K\"ahler-Ricci solitons.
\end{cor}

\subsection{Four dimensional log nondegenerate GKRS}

Remarkably, in dimension 4, we establish in Proposition~\ref{p:4dlagrangian} and Lemma~\ref{p:spanprop} that the necessary conditions (1)-(4) of Theorem~\ref{t:Atypeconstruction} hold a priori, thus exhibiting \emph{all} rank $2$ GKRS as $A$-deformations of an underlying toric KRS over a dense open subset.  In the statement below $\gs$ is the real Poisson tensor canonically associated to a GK manifold (cf. Proposition \ref{p:HitchinPoisson} below).

\begin{thm} \label{t:4dmainthm}  Let  $(M^4, g, I, J, f)$ be a rank 2 GKRS with $\gs$ not vanishing identically.  Then $(g, I, J)$ is an $A$-deformation over the dense open subset $\mathring{M}$ of points where $\sigma \neq 0$ of a locally toric KRS $(g_K, J_K, f_K)$.  If, furthermore $(g, I, J)$ is symplectic-type on $M$ and $g$ is complete, then $(g_K, J_K, f_K)$ is a complete non-Ricci-flat KRS on $M$.
\end{thm}

\begin{rmk}
\begin{enumerate}
\item Combining Theorem \ref{t:4dmainthm} with the local theory~\cite{bryant-KRS} of  gradient toric K\"ahler-Ricci solitons, this provides a complete local classification of such GKRS's. 
\item In the case where $\gs$ does vanish identically, the tangent bundle admits a splitting into a direct sum of line bundles (cf. \cite{ApostolovGualtieri}).  The classification of rank $2$ solitons in this setting remains an open question.
\item We emphasize the global assumption in Theorem~\ref{t:4dmainthm} that the GKRS is symplectic-type. For instance, the universal cover  $\C^2\setminus \{0\}$ of a diagonal Hopf complex surface is an example of complete GKRS which is not symplectic-type.
\item Note that any four-dimensional locally toric KRS on $M$ which is not Ricci-flat gives rise to a symplectic-type \emph{rank 2} GKRS on $M$.  In particular Corollary~\ref{c:type-A} specializes in dimension $4$ as follows: 
Fix $\Gamma=\Z_{k+1}\subset (\C^* \times \C^*) \cap {\rm SL}(2, \C)$ be a finite cyclic subgroup and let $M^4_k$  be (the unique) toric crepant resolution of the (Calabi-Yau cone) $\C^2/\Z_{k+1}$.  Then $M^4_k$ admits rank $2$ GKRS.  
\item By Theorems~\ref{t:completeness} and \ref{t:4dmainthm}, the $A$-deformations of the multi-Euguchi-Hanson and the multi-Taub-NUT K\"ahler Ricci-flat metrics on $M^4_k$ define complete toric GKRS of rank 1. These examples appear as special cases in the classification obtained in \cite{SU2}, see Proposition~\ref{p:toric-rank1} below, but the observation that they are naturally linked to the gravitational instantons on these spaces is new.
\end{enumerate}
\end{rmk}

To establish Theorem~\ref{t:4dmainthm}, we recast nondegenerate rank 2 solitons using the Gibbons-Hawking ansatz of \cite{SU2}.  While \cite{SU2} focuses on the rank $1$ case, it gives a general description of four-dimensional GK structures with a Killing symmetry in terms of reduced data on a three-manifold.  In the rank $2$ case we thus obtain two distinct reductions, which refines the structure considerably.  We first derive an expression for the generalized Ricci potential $\Phi$ using mixed momentum coordinates for the two symplectic forms $F_{\pm}$.  This in turn gives an explicit description of the horizontal components of the two canonical symmetries, in turn yielding a direct local proof of the key property that the torus $\tor = \spn_{\mathbb R} \{X_I, X_J\}$ satisfies $I \tor = J \tor$. Moreover we show the that four-dimensional symplectic-type GKRS satisfy $F_+(X_I, X_J) = 0$ in general.  Thus we are in a position to apply Theorem \ref{t:Atypeconstruction} to conclude Theorem \ref{t:4dmainthm}.   However we also provide a direct, elementary, derivation of the underlying symplectic potentials associated to $F_{\pm}$ in this setting.  These are shown to satisfy a drift Monge-Amp\`ere equation, and moreover are Legendre dual to each other, giving a concrete manifestation of the general Legendre duality discussion in \S \ref{s:Legendre}.

The constructions in~\cite{cao1996existence,apostolov2023hamiltonian},  as well as the products of two cigar solitons on $\C$,   or of a cigar soliton with a flat $\C$,  give rise to many more explicit non Ricci-flat complete toric KRS's on $M^4_0=\C^2$, which by our results  lead to complete rank 2 GKRS on $\R^4$.
Furthermore, these complete toric KRS's on $\C^2$ define orbifold KRS metrics on the singularity $\C^2/\Z_{k+1}$. The gluing result~\cite{Biquard-Macbeth} yields the existence of a complete, 
non Ricci-flat toric KRS's resolution (and hence of complete rank 2 GKRS's) on $M^4_k, \, k\geq 1$,   which have the asymptotic behavior of Cao's  ${\rm U}(2)$-invariant example on $\C^2/\Z_{k+1}$.  It appears plausible to us that all the singular KRS's models on $\C^2/\Z_{k+1}$ can be equivariantly resolved by complete toric KRS's on $M^4_k, \, k \geq 1$,  keeping the asymptotic at infinity of the given singular model metric; furthermore, it is natural to ask whether or not  these examples will exhaust all complete $4$-dimenional toric KRS's,  and hence also the complete symplectic-type GKRS of rank 2, corresponding to the unbounded Delzant polytopes as in Theorem~\ref{t:completeness}.

\section{Local Theory of Toric Generalized K\"ahler-Ricci solitons} \label{s:toricGKRS}

We begin this section with a review of fundamental properties of generalized K\"ahler-Ricci solitons, including (mostly known) properties of the a priori symmetries.  We then recall the basic properties of $A$-deformations of toric K\"ahler structures, and then give the proof of Theorem \ref{t:Atypeconstruction}.  We end this section with a discussion of the role of Legendre duality.

\subsection{Definitions and a priori symmetries}

\begin{defn} \label{d:PCS} 
\begin{enumerate}
    \item 
A quadruple $(M, g, H, f)$ of a Riemannian manifold with closed three-form $H$ and smooth function $f$ is a \emph{gradient generalized Ricci soliton} if
\begin{align*}
    \Rc^{\N^+} + \N^+ \N^+ f = 0, \qquad \N^+ = \N + \tfrac{1}{2} g^{-1} H.
\end{align*}

\item A pluriclosed structure $(M, g, J)$ is a \emph{gradient Ricci soliton} if there exists a function $f$ such that the triple $(g, - d^c \gw, f)$ defines a gradient generalized Ricci soliton.
\vskip 0.1in
\item A generalized K\"ahler structure $(M, g, I, J)$ is a \emph{generalized K\"ahler-Ricci soliton} (GKRS) if there exists a function $f$ such that the triple $(g, - d^c_I \gw_I, f)$ defines a gradient generalized Ricci soliton.  An elementary remark is that this holds if and only if also $(g, - d^c_J \gw_J, f)$ defines a gradient generalized Ricci soliton.
\end{enumerate}
\end{defn}

\noindent We next recall an a priori symmetry present on pluriclosed solitons:

\begin{prop} \cite{SU2} \label{p:PCsolitonsymm} Given a pluriclosed manifold $(M, g, J)$, the data $(g, - d^c \gw, f)$ is a gradient soliton if and only if the vector field
\begin{align*}
V = \tfrac{1}{2} (\theta^{\sharp} - \N f)
\end{align*}
satisfies
\begin{align*}
    L_V J = 0, \qquad L_{JV} g = 0, \qquad L_V \gw = \rho_B^{1,1}.
\end{align*}
\end{prop}

Turning now to generalized K\"ahler structures, we first recall the canonically associated Poisson tensor \cite{AGG,HitchinPoisson,PontecorvoCS}.  We say that a generalized K\"ahler structure is \emph{nondegenerate} if this Poisson tensor is invertible.

\begin{prop} \label{p:HitchinPoisson} (\cite{AGG, HitchinPoisson,PontecorvoCS}) Given $(M^{2n}, g, I, J)$ a generalized K\"ahler structure, the tensor
\begin{align*}
    \gs :=&\ \tfrac{1}{2} [I, J] g^{-1}
\end{align*}
is a real Poisson tensor.  Furthermore,
\begin{align*}
    \gs \in \Lambda^{2,0 + 0,2}_I(TM) \cap \Lambda^{2,0 + 0,2}_J (TM), \quad \delb_I \gs_I^{2,0} = 0, \quad \delb_J \gs_J^{2,0} = 0.
\end{align*}
\end{prop}

\begin{defn} \label{d:Riccipot} Let $(M^{2n}, g, I, J)$ be a generalized K\"ahler structure defined by closed spinors $\psi_1, \psi_2$.  Define the \emph{partial Ricci potentials} by
\begin{align*}
    \Psi_i = - \log \frac{(\psi_1,\bar{\psi}_1)}{dV_g}.
\end{align*}
Define the \emph{Ricci potential} by
\begin{align*}
    \Phi = \Psi_1 - \Psi_2 = - \log \frac{ (\psi_1, \bar{\psi}_1)}{(\psi_2, \bar{\psi}_2)}.
\end{align*}
\end{defn}

\begin{prop} (\cite{apostolov2022generalized} Proposition 4.3) \label{p:transgression} Let $(M^{2n}, g, I, J)$ be a generalized K\"ahler structure defined by closed spinors $\psi_1, \psi_2$.  Then
\begin{gather*}
    (I - J) d \Psi_1 = I \theta_I - J \theta_J, \qquad (I + J) d \Psi_2 = I \theta_I + J \theta_J\\
    \gs d \Phi = \theta_I^{\sharp} - \theta_J^{\sharp}, \qquad 
    \rho_I = - \tfrac{1}{2} d J d \Phi, \qquad \rho_J = - \tfrac{1}{2} d I d \Phi.
\end{gather*}
\end{prop}

The next proposition summarizes the nature of the a priori symmetries present on a GKRS.  The result is essentially known, though the proof is spread over many works so we reproduce some simplified arguments here for convenience.

\begin{prop} \label{p:GKRSsymmetries} Let $(M^{2n}, g, I, J, f)$ be a GKRS.  Define the vector fields
\begin{align*}
    V_I &=  \tfrac{1}{2} (\theta_I^{\sharp} - \N f), \qquad  \, \, \, \, \, \,  V_J = \tfrac{1}{2} (\theta_J^{\sharp} - \N f), \\
    X_I = J V_I &=\tfrac{1}{2} (I\theta_I^{\sharp} - I\N f), \qquad  X_J = J V_J = \tfrac{1}{2} (J\theta_J^{\sharp} - J\N f).
\end{align*}
Then 
\begin{align*}
0 =&\ L_{V_I} I = L_{V_J} J,\\
0 =&\ L_{V_I} \gs = L_{V_J} \gs,\\
0 =&\ L_{X_I} I = L_{X_I} J = L_{X_I} g,\\
0 =&\ L_{X_J} I = L_{X_J} J = L_{X_J} g,\\
0 =&\ [V_I, X_J] = [X_I, X_J] = [V_J, X_I].
\end{align*}
\begin{proof} By definition the data $(g, - d^c_I \gw_I, f)$ defines a pluriclosed generalized Ricci soliton, therefore by Proposition \ref{p:PCsolitonsymm} we have
\begin{align*}
    L_{V_I} I = 0, \quad L_{I V_I} g = 0.
\end{align*}
As remarked above the data $(g, - d^c_J \gw_J, f)$ is also a pluriclosed gradient soliton, thus again by Proposition \ref{p:PCsolitonsymm} we have
\begin{align*}
    L_{V_J} J = 0, \quad L_{J V_J} g = 0.
\end{align*}
Next, arguing as in (\cite{ASU3} Proposition 4.6), we obtain
\begin{align*}
    L_{V_I} \gs = L_{I V_I} \gs = L_{V_J} \gs = L_{J V_J} \gs = 0.
\end{align*}
Now we turn to arguments from \cite{ASUScal}.  In particular, defining $\pi_{\mathbb J} = \tfrac{1}{2} (I + J) g^{-1}$, from (\cite{ASUScal} Lemma 3.6) we can express
\begin{align*}
    Z = I V_I + J V_J = \tfrac{1}{2} (I \theta_I^{\sharp} + J \theta_J^{\sharp}) - \pi_{\JJ} df = \divg_{e^{-f} dV_g} \pi_{\JJ}.
\end{align*}
It follows using this and the above identities that
\begin{align*}
    0 =&\ L_{Z} \pi_{\JJ} = L_{I V_I + J V_J} \left( \tfrac{1}{2} g^{-1} (I + J) \right) = \tfrac{1}{2} g^{-1} \left( L_{J V_J} I + L_{I V_I} J \right).
\end{align*}
Thus furthermore
\begin{align*}
    0 =&\ L_{I V_I} (\gs g) = L_{I V_I} \left( \tfrac{1}{2} g^{-1} [I, J] \right) = \tfrac{1}{2} g^{-1} [ I, L_{I V_I} J] = \tfrac{1}{2} g^{-1} [ L_{J V_J} I, I] = g^{-1} \left( L_{J V_J} I \right) I.
\end{align*}
Hence $L_{J V_J} I = 0$, and similarly $L_{I V_I} J = 0$.
\end{proof}
\end{prop}

\begin{defn} \label{d:rank} Let $(M^{2n}, g, I, J)$ be a GKRS.  Define the \emph{rank} of the soliton as
\begin{align*}
    \mbox{rank}(M) := \sup_M \dim \spn \{X_I, X_J\}.
\end{align*}
\end{defn}

\begin{rmk} Notice that by (\cite{apostolov2024classification} Proposition 2.6) a compact GKRS has rank zero if and only if it is K\"ahler Calabi-Yau.  The four-dimensional rank $1$ case was classified in \cite{SU2}.
\end{rmk}

We record one further key property in four dimensions.  We note that a similar vanishing result was derived in all dimensions in (\cite{ASU3} Proposition 4.6), although the argument requires compactness, which we do not want to assume here.

\begin{prop} \label{p:4dlagrangian} Let $(M^4, g, I, J)$ be a GKRS.  Then where defined, the symplectic $2$-forms $F_{\pm}=-2g(I \pm J)^{-1}$ satisfy
\begin{align*}
    F_+(X_I, X_J) = F_- (X_I, X_J) = 0.
\end{align*}
\begin{proof}
We first observe that in four dimensions, the pluriclosed soliton equation $d^* (e^{-f} H) = 0$ is equivalently expressed as
    \begin{align*}
        0 = d \theta - df \wedge \theta.
    \end{align*}
    As $d \theta$ is primitive since the metric is Gauduchon, it follows that $g(\theta^{\sharp}, J \N f) = 0$.  In the GKRS setting we thus obtain
    \begin{align*}
        0 =&\ g(\theta_I^{\sharp}, I \N f) = g(\theta_J^{\sharp}, J \N f).
    \end{align*}
        In $4$ dimensions, we have $(I+J)^2 = -2(p+1) {\rm Id}$ where $p = - \frac{1}{4}{\rm trace}(IJ)$ is the angle function, so that $F_+ = \frac{1}{(1+p)}g(I+J)$.  Thus using further that $\theta_I = - \theta_J$,
\begin{align*}
F_+(X_I, X_J) =&\ \frac{1}{(1+p)}g\big((I+J)X_I, X_J\big) =  \frac{1}{(1+p)}\big(g(X_I, V_J) - g(V_I, X_J)\big)\\
                    &= \frac{2}{(1+p)}\left( g(\theta_I^{\sharp}, (I + J) \N f) \right)=0,
\end{align*}
as claimed.  The computation for $F_-(X_I, X_J)$ is similar.
\end{proof}
\end{prop}

\subsection{The generalized K\"ahler \texorpdfstring{$A$}{A}-deformation of a K\"ahler toric structure} \label{ss:Atransform}

We recall fundamental aspects of toric generalized K\"ahler structures of symplectic-type following \cite{boulanger2019toric, wang2020toric}, emphasizing first the local aspects of the theory.  

\begin{defn}\label{d:Kahler-locally-toric}
Let $(\mathring{M}, F)$ be a symplectic $2n$-dimensional manifold endowed with an $F$-compatible K\"ahler metric $(g_K, J_K)$, 
and $n$ commuting Killing vector fields $\{ X_1, \ldots, X_n \}$. We let $\tor:={\rm span}_{\R} \{ X_1, \ldots, X_n \}$  be the corresponding Lie algebra. We say that the K\"ahler manifold $(\mathring{M}, F, g_K, J_K, \tor)$  is \emph{locally toric} if  at each point $p\in \mathring{M}$:
\begin{enumerate}
\item $\tor_p:={\rm span}_{\R} \{ X_1(p), \ldots, X_n(p) \} \subset T_p\mathring{M}$ is a $n$-dimensional distribution;
\item $F(X_i(p), X_j(p))=0$.
\end{enumerate}
\end{defn}

The above conditions and the integrability of $J_K$ give rise to the splitting of the tangent space
\[ T_p\mathring{M} = \tor_p \oplus J_K \tor_p\]
as the direct sum of two integrable $F$-Legendrian distributions, and to a local coordinate system $(t_j, y_i)$ associated to the commuting vector fields $\{ X_j, J_K X_i\}$. Furthermore, we can introduce (local) momentum coordinates by $F(X_j)= - d\mu_j$, so that we  also have angular/momentum  coordinates $\{t_j, \mu_i\}$ with respect to which
\begin{equation}\label{momentum-angle}
\begin{split}
& F  = \sum_{j=1}^n d\mu_i \wedge dt_j, \qquad
X_j   = \frac{\partial}{\partial t_j},  \\
& J_K = \sum_{i,j=1}^n \Big( {\bf G}^K_{ij}(\mu) d\mu_i \otimes \frac{\partial}{\partial t_j}- {\bf H}^K_{ij}(\mu) dt_i \otimes  \frac{\partial}{\partial \mu_j}\Big),
\end{split}
\end{equation}
where ${\bf G}^K(\mu):=({\bf H}^K(\mu))^{-1}$ is a symmetric positive definite matrix with smooth coefficients which satisfies the Hessian condition ${\bf G}^K_{ij,k}={\bf G}^{K}_{ik,j}$,  where $,k$ stands for the partial derivative with respect to  the variable $\mu_k$.    The K\"ahler metric is then
\begin{equation}\label{g_K}
g_K  = \sum_{i,j=1}^n \Big({\bf G}^K_{ij}(\mu) d\mu_i d\mu_j + {\bf H}^K_{ij}(\mu) dt_i dt_j\Big).
\end{equation}

Following \cite{boulanger2019toric}, we now describe the construction of $A$-deformations of the structure above.  Fix $A\in \Lambda^2 \tor$, and then define a new $\tor$-invariant almost complex structure $J$  on $\mathring{M}$ by
\begin{equation}\label{boulanger}
J:= \sum_{i,j=1}^n \Big( {\bf \Psi}_{ij}(\mu) d\mu_i \otimes \frac{\partial}{\partial t_j}- {\bf \Psi}^{ij}(\mu) dt_i \otimes  \frac{\partial}{\partial \mu_j}\Big),
\end{equation}
where
\[ {\bf \Psi}_{ij}(\mu):= ({\bf G}^K_{ij}(\mu) + A_{ij}), \qquad {\bf \Psi}^{ij}(\mu) = ({\bf \Psi}(\mu))^{-1}_{ij}.\]
By \cite[Theorem~6]{boulanger2019toric},  $J$ is integrable if and only if
\begin{equation*}
{\bf \Psi}_{ij, k} ={\bf  \Psi}_{ik,j}.
\end{equation*}
Furthermore, the symplectic form $F$ tames $J$ if and only if the the symmetric tensor
\begin{equation} \label{g}
\begin{split}
g (\cdot, \cdot):=  -F(J\cdot, \cdot)^{\rm sym}
   = \sum_{i,j=1}^m \Big( ({\bf \Psi}^{\rm s})_{ij}(\mu) d\mu_i d\mu_j + ({\bf \Psi}^{-1})^{\rm s}_{ij}(\mu) dt_idt_j\Big)
   \end{split}
\end{equation}
is positive definite on $\mathring{M}$.  This condition is independent of $A$, as can be seen from the following general identities for a nondegenerate matrix ${\bf \Psi}$ which follow readily from the polar decomposition (the superscripts denote the symmetric and antisymmetric components, and transpose):
\begin{equation}\label{linear-algebra}
\begin{split}
& {\bf \Psi} \big({\bf \Psi}^{-1}\big)^{\rm s} {\bf \Psi}^{\rm T} = {\bf \Psi}^{\rm s}, \, \, {\bf \Psi} \big({\bf \Psi}^{-1}\big)^{\rm a} {\bf \Psi}^{\rm T} = -{\bf \Psi}^{\rm a}; \\
& {\bf \Psi}^{\rm s}\big({\bf \Psi}^{-1}\big)^{\rm a} + {\bf \Psi}^{\rm a} \big({\bf \Psi}^{-1}\big)^{\rm s} =0, \, \,  {\bf \Psi}^{\rm s}\big({\bf \Psi}^{-1}\big)^{\rm s} + {\bf \Psi}^{\rm a} \big({\bf \Psi}^{-1}\big)^{\rm a} ={\rm Id}.
\end{split}
\end{equation}
The crucial point of this construction is that the almost-complex structure $I$ defined by
\begin{equation}\label{eq:IJconj}
F(JX, Y)= -F(X, IY)
\end{equation}
is also integrable, and the symmetric tensor $g$ defined above is $I$-invariant.  In the angular-momentum local coordinate system $(t, \mu)$ the almost complex structure $I$ is expressed
\begin{equation}\label{I-matrix}
I = \sum_{i,j=1}^n \Big( ({\bf \Psi}^{\rm T})_{ij}(\mu) d\mu_i \otimes \frac{\partial}{\partial t_j}- \big({\bf \Psi}^{\rm T}\big)^{ij}(\mu) dt_i \otimes  \frac{\partial}{\partial \mu_j}\Big).
\end{equation}
We furthermore define a useful associated two-form $\be$ on $\mathring{M}$ via 
\begin{equation}\label{b} 
\begin{split} \be(\cdot, \cdot) &:= - \big(F(J\cdot, \cdot)\big)^{\rm skew} \\ &=  - \tfrac{1}{2} \sum_{i,j=1}^n \Big(A_{ij} d\mu_i \wedge d\mu_j  + ({\bf \Psi}^{-1})^{\rm a}_{ij}(\mu) dt_i\wedge dt_j\Big) \\ &= - \Big(\sum_{i,j=1}^n A_{ij}d\mu_i \wedge d\mu_j \Big)^{(2,0) + (0,2)}_J.
\end{split}
\end{equation}
\noindent The fundamental result of Boulanger (\cite{boulanger2019toric} Theorem 4) is that the data $(F, g, I, J)$ defines a symplectic-type generalized K\"ahler.  We codify this construction with a definition:

\begin{defn}\label{d:diagonal}
The symplectic-type generalized K\"ahler structure $(F, g, I, J)$ obtained from a toric $F$-compatible K\"ahler structure $(g_K, J_K)$ via the above construction is an \emph{$A$-deformation} of $(g_K, J_K)$. 
\end{defn}

We recall the following characterization of $A$-deformations (\cite{boulanger2019toric}, Prop. 4).

\begin{prop}\label{p:GK-locally-toric} Suppose $(F, g, I, J)$ is a symplectic-type generalized K\"ahler
structure on $(\mathring{M}, F)$ which is locally toric in the sense that there are  Killing fields $\{ X_1, \ldots, X_n \}$ such that the distribution $\tor={\rm span}_{\R}\{ X_1, \ldots, X_n \}$ satisfies the conditions of Definition~\ref{d:Kahler-locally-toric}. Then $(g, I, J)$ is obtained by an $A$-deformation of a locally toric K\"ahler metric $(F, g_K, J_K, \tor)$ if and only if at each point $p\in \mathring{M}$ 
\[ I \tor_p = J\tor_p.\]
\end{prop}

We close this section with some further fundamental points and computations used below.  First notice that the fundamental forms $\omega_I := gI, \omega_J := gJ$ are written in the coordinate system $(t, \mu)$ as
\begin{equation}\label{omega}
\begin{split}
\omega_J &= \sum_{i,j=1}^m\Big({\bf \Psi}^{\rm s}\big({\bf \Psi}^{\rm T}\big)^{-1}\Big)_{ij}(\mu)d\mu_i\wedge dt_j, \quad  \omega_I = \sum_{i,j=1}^m\Big({\bf \Psi}^{\rm s}{\bf \Psi}^{-1}\Big)_{ij}(\mu)d\mu_i\wedge dt_j.
\end{split}
\end{equation}
Furthermore, the corresponding holomorphic Poisson structure~\cite{HitchinPoisson} $\sigma_J:= \pi_{\Lambda^{2,0}_J} \gs$ is expressed (cf. \cite{ASU2} Proposition 3.9)
\begin{gather} \label{J-poisson}
\sigma_J=  -2\sum_{i,j=1}^m A_{ij} \big(X_i -\sqrt{-1} JX_i\big) \wedge \big(X_j - \sqrt{-1}JX_j\big),
\end{gather}
which provides a geometric interpretation of $A$.  In particular this shows that the Poisson tensor $\gs$ is nondegenerate if and only if the matrix $A$ is nondegenerate.  Furthermore, observe that equation \eqref{boulanger} does not depend on the choice of local angular/momentum
coordinates, as we can equally write for any $\xi \in J_K\tor_p$, 
\[ J (\xi) = J_K(\xi) +AF (\xi)  \qquad I(\xi) = J_K(\xi) - A F (\xi),   \]
where $A= \sum_{i,j=1}^n A_{ij} X_i \wedge X_j$ is viewed as a real Poisson tensor on $\mathring{M}$,  and $J_K, I, J$, $AF$ are linear maps which interchange $J_K \tor_p$ and $\tor_p$.  Finally we give an explicit relationship between the Riemannian metrics underlying an $A$-deformation.

\begin{lemma} \label{l:Atypemetriccomparison}
Given the setup above, on $\mathring{M}$ one has
\begin{equation}\label{extension}
\begin{split}
g- g_K =&\ -g\left(A\frac{(I+ J)}{2}d\mu, A\frac{(I+ J)}{2}d\mu\right),
\end{split}
\end{equation}
and
\begin{equation}\label{extension-K}
\begin{split}
g-\be-g_K &=\left\langle d\mu,  A, d\mu \right\rangle - \left\langle J_K d\mu, A\left({\rm Id} +\left({\bf \Psi}^{s}\right)^{-1}A\right)^{-1},  J_K d\mu \right\rangle.
\end{split}    
\end{equation}

\begin{proof} 
First observe the consequence of the second line in \eqref{linear-algebra}
\[ 
\begin{split} 
{\bf \Psi}^{s} \left({\bf \Psi}^{-1}\right)^s {\bf \Psi}^{s} &= 
\left({\rm Id} - A({\bf \Psi}^{-1})^{a}\right){\bf \Psi}^s ={\bf \Psi}^{s}-A({\bf \Psi}^{-1})^{a}{\bf \Psi}^{s} = {\bf \Psi}^{s} + A({\bf \Psi}^{-1})^s A.
\end{split}\]
Using this and the expressions for $J$ and $I$ from \eqref{boulanger} and \eqref{I-matrix} we obtain
\begin{equation*}
\begin{split}
g- g_K &= \langle dt, \left({\bf \Psi}^{-1}\right)^{s} -\left({\bf \Psi}^s\right)^{-1}, dt \rangle \\
&= \left\langle \frac{(I+ J)}{2}d\mu, {\bf \Psi}^s\left(\left({\bf \Psi}^{-1}\right)^{s} -\left({\bf \Psi}^s\right)^{-1}\right) {\bf \Psi}^{s}, \frac{(I+ J)}{2}d\mu \right\rangle \\
&= \left\langle \frac{(I+ J)}{2}d\mu, A\left ({\bf \Psi}^{-1}\right)^{s}A, \frac{(I+ J)}{2}d\mu \right\rangle  \\
&= \sum_{i,j=1}^n \Big(A({\bf \Psi}^{-1})^s A\Big)_{ij}\left(\frac{I+J}{2}\right) d\mu_i\left(\frac{I+J}{2}\right)d\mu_j \\
&= -g\left(A\frac{(I+ J)}{2}d\mu, A\frac{(I+ J)}{2}d\mu\right),
\end{split}
\end{equation*}
giving (\ref{extension}).  Taking into account equation \eqref{b}, we compute
\begin{equation*}
\begin{split}
g-\be-g_K &=  \left\langle d\mu,  A, d\mu \right\rangle +\left\langle J_K d\mu, {\bf \Psi}^s\left({\bf \Psi}^{-1}-\left({\bf \Psi}^s\right)^{-1}\right) {\bf \Psi}^{s},  J_K d\mu \right\rangle   \\
&=\left\langle d\mu,  A, d\mu \right\rangle - \left\langle J_K d\mu, A\left({\rm Id} +\left({\bf \Psi}^{s}\right)^{-1}A\right)^{-1},  J_K d\mu \right\rangle,
\end{split}    
\end{equation*}
giving (\ref{extension-K}).
\end{proof}

\end{lemma}

\begin{rmk} Notice that tensor on the right hand side of (\ref{extension}) is understood as follows: if $v, w \in T_pM$, then $\left(\frac{(I+ J)}{2} d\mu\right)_p(v)$  and $\left(\frac{(I+ J)}{2} d\mu\right)_p(w)$ are viewed as elements of $\tor^*$ on which $A$ acts to produce elements of $\tor$, and hence fundamental vector fields $X_v, X_{w}$ on $M$. Then
\[ g\left(A\frac{(I+ J)}{2}d\mu, A\frac{(I+ J)}{2}d\mu\right)_p(v, w):= g_p(X_v(p), X_{w}(p)).\]  Similar remarks apply to the right hand side of (\ref{extension-K}).
\end{rmk}

\noindent Finally we record a technical lemma classifying pluriharmonic functions on the moment image.
\begin{lemma}\label{pluriharmonic} If $(F, g, I, J)$ is an $A$-deformation of a locally toric K\"ahler structure with ${\bf G}^K(\mu)= {\rm Hess}(u)$ and $A\in \wedge^2 \tor$, then the functions  $u_j, \, j=1, \ldots, n$ are $J_K$-pluriharmonic,  the functions $u_{, j} + \sum_{k=1}^n A_{kj}\mu_k, \, j=1, \ldots, n$ are $J$-pluriharmonic, and the functions $u_{, j} - \sum_{k=1}^n A_{kj}\mu_k, \, j=1, \ldots, n$ are $I$-pluriharmonic.  Moreover, these are the only $\tor$-invariant pluriharmonic functions on the respective complex manifolds.
\end{lemma}
\begin{proof} The fact that the relevant functions are pluriharmonic can be directly computed (cf. \cite{ASU2} (3.12)).  To see that these are the only such, fix for instance $\varphi$, a $\tor$-invariant pluriharmonic function with respect to $J_K$. As it is $\tor$-invariant, $\varphi = \varphi(y)$, with $y = \N u$.  Since $y$ is $J_K$-pluriharmonic, we have
\begin{align*}
    0 = d d^c_{J_K} \varphi = \Hess(\varphi)_{ij} d y_i \wedge d t_j,
\end{align*}
hence $\Hess(\varphi) = 0$ and the claim follows.
\end{proof}

\subsection{Local equivalence of K\"ahler and generalized K\"ahler-Ricci solitons}

Here we establish the local equivalence between toric K\"ahler-Ricci solitons and toric symplectic-type generalized K\"ahler-Ricci solitons through the $A$-deformation.  The general setup through this subsection is to fix a local toric K\"ahler structure $(F, g_K, J_K)$ compatible with $F$ and an element $A\in \Lambda^2 \tor$, and then relate certain curvature properties of $(g_K, J_K)$ to the corresponding curvature properties of the generalized K\"ahler $A$-deformation $(F, g, I, J)$. Our considerations are local, so we can represent the K\"ahler data $(g_K, J_K)$ with a strictly convex smooth function $u(\mu)$,  unique up to an affine linear additive function in momenta, so that ${\bf G}^K(\mu) = {\rm Hess}(u)$ in \eqref{momentum-angle}.

We say that the locally toric K\"ahler metric $(g_K, J_K)$ is a \emph{locally toric KRS} if
\eqref{d:KRS} holds for some $\tor$-invariant function $f_K$.  It follows that $f_K = - \IP{b, \mu}$ for some $b \in \tor$.  Furthermore, there is a reduction to a drift Monge-Amp\`ere equation:
\begin{lemma}\label{l:KRS} (\cite{apostolov2023hamiltonian} Lemma 2.10, \cite{Donaldson-survey} Eq. (11)) ${\bf G}^K={\rm Hess}(u)$ defines a KRS  $(F, g_K, J_K)$ on $\mathring{M}$ with potential function $f_K = -\IP{b, \mu}, \, \, b \in \tor$ if and only if there is a constant $c\in \tor^*$ such that
\begin{gather} \label{f:KRS}
\langle b , x \rangle = \langle c, \nabla u\rangle + \log \det {\rm Hess}(u) + const.
\end{gather}
\end{lemma}
\begin{proof} Let $b\in \tor$ and $f_K=-\IP{b, \mu}$. Recall that (see e.g. \cite{abreu2012scalar,apostolov-notes}) the Ricci form of $(g_K, J_K)$  is given by 
\[ \rho_{g_K} = \frac{1}{2} dd^c_{J_K} \log \det {\rm Hess}(u).\]
Therefore,  as $J_K \nabla f_K$ is Killing, the KRS equation \eqref{d:KRS} becomes
$\rho_K = -\frac{1}{2} dd^c_{J_K} f_K$, which in turn is equivalent to $\log \det {\rm Hess}(u) - \langle \mu, b \rangle$ being a $\tor$-invariant $J_K$-pluriclosed function. Any such function is of the form  $-\langle c, \nabla u \rangle + const$ by Lemma~\ref{pluriharmonic}. \end{proof}

\begin{rmk} The real additive constant  in \eqref{f:KRS} can be set to be zero (as in \cite{apostolov2023hamiltonian}) by choosing a basis of $\tor$ (and hence $\tor^*$) for the definition of $\det\Hess(u)$. 
\end{rmk}

\noindent To make the connection with generalized K\"ahler-Ricci solitons, we first record a local expression for the Bismut-Ricci forms of an $A$-deformation of a toric K\"ahler structure.

\begin{lemma}\label{Bismut-Ricci} (\cite{ASU2} Lemma 3.10) Let $(F,g, I,J)$ be an $A$-deformation of a locally toric K\"ahler structure $(g_K, J_K)$, where ${\bf G}^K= {\rm Hess}(u)$.  Then the Bismut-Ricci forms $\rho_J$ and $\rho_I$ are given by 
\begin{equation}\label{e:Ricci-Bismut} 
 \rho_J = \tfrac{1}{2} d d^c_I \log \det {\rm Hess}(u), \qquad \rho_I = \tfrac{1}{2}dd^c_J  \log \det {\rm Hess}(u). 
\end{equation}                 
\end{lemma}
\noindent We derive below a key formula for the Lee forms:
\begin{lemma}\label{Lee-forms} Let $(F,g, I,J)$ be an $A$-deformation of a locally toric K\"ahler structure $(g_K, J_K)$, where ${\bf G}^K= {\rm Hess}(u)$.  Then, the Lee forms $\theta_I$ and $\theta_J$ of the Hermitian structures $(g, I)$ and $(g, J)$ are given by
\[  
\begin{split}
\theta_J =&\ \tfrac{1}{2}d\log \det {\rm Hess}(u) -d \log \det {\bf \Psi}   - \tfrac{1}{2}  d\log \det {\rm Hess}(u) \circ IJ,  \\
\theta_I =&\ \tfrac{1}{2}d\log \det {\rm Hess}(u)  -d \log \det {\bf \Psi}  - \tfrac{1}{2} d\log \det {\rm Hess}(u) \circ JI.
\end{split}\]
\end{lemma}
\begin{proof}
The Lee form $\theta_J$ is equivalently expressed
\begin{equation}\label{e:theta}
  \theta_J = \omega_J \lrcorner (d\omega_J)= -\big\langle (d\omega_J), (\omega_J)^{-1}\big\rangle.
  \end{equation}
Using formula  \eqref{omega}, we find
 \[
 \begin{split}
 d\omega_J  &=  \sum_{i,j,r=1}^m \Big({\bf \Psi}^{\rm s}\big({\bf \Psi^{\rm T}}\big)^{-1} \Big)_{ij, r} d\mu_i \wedge  dt_j\wedge d\mu_r, \\
 ({\omega_J})^{-1}  &= -\sum_{i,j=1}^m \Big({\bf \Psi}^{\rm T}\big({\bf \Psi}^{\rm s}\big)^{-1}\Big)_{ji} \frac{\partial}{\partial \mu_i} \wedge \frac{\partial}{\partial t_j}.
 \end{split} 
 \]
Substituting in \eqref{e:theta}, we get
 \begin{equation}\label{long-theta}
 \begin{split}
 \theta_J =&\  {\rm tr}\Big(\big({\bf \Psi}^{\rm T} ({\bf \Psi}^{\rm s})^{-1}\big)  \big({\bf \Psi}^{\rm s}{(\bf \Psi^{\rm T})}^{-1} \big)_{, r}\Big) d\mu_r -  \Big(\big({\bf \Psi}^{\rm s}({\bf \Psi}^{\rm T})^{-1}\big)_{, r}\big({\bf \Psi}^{\rm T} ({\bf \Psi}^{\rm s})^{-1}\big)\Big)_{ir}d\mu_i \\
=&\ d \log \det \big(({\bf \Psi}^{\rm T})^{-1} {\bf \Psi}^{\rm s}\big) - \Big(A \big(({\bf \Psi}^{\rm T})^{-1}\big)_{,r} \big({\bf \Psi}^{\rm T} ({\bf \Psi}^{\rm s})^{-1}\big)\  \Big)_{ir} d\mu_i  \\
=&\ d \log \det \big(({\bf \Psi}^{\rm T})^{-1} {\bf \Psi}^{\rm s}\big) + \Big(A({\bf \Psi}^{\rm T})^{-1}({\bf \Psi}^{\rm T})_{,r} ({\bf \Psi}^{\rm s})^{-1} \Big)_{ir} d\mu_i  \\
=& -d\log \det ({\bf \Psi}) + d\log \det {\rm Hess}(u) + \Big(A({\bf \Psi}^{\rm T})^{-1}({\bf \Psi}^{\rm s})_{,r} ({\bf \Psi}^{\rm s})^{-1} \Big)_{ir} d\mu_i  \\
=& -d\log \det ({\bf \Psi}) + d\log \det {\rm Hess}(u) + \big(A({\bf \Psi}^{\rm T})^{-1}\big)_{i\ell}\big(({\bf \Psi}^{\rm s})_{,r}\big)_{\ell j} \big(({\bf \Psi}^{\rm s})^{-1}\big)_{jr}  d\mu_i \\
=& -d\log \det ({\bf \Psi}) + d\log \det {\rm Hess}(u) \\
&+ \tfrac{1}{2} \big(({\bf \Psi} -{\bf \Psi}^{\rm T})({\bf \Psi}^{\rm T})^{-1}\big)_{i\ell}\big(({\bf \Psi}^{\rm s})_{,\ell}\big)_{rj} \big(({\bf \Psi}^{\rm s})^{-1}\big)_{jr}  d\mu_i \\
=& -d\log \det ({\bf \Psi}) + \tfrac{1}{2} d \log \det {\rm Hess}(u) -\tfrac{1}{2}  d\log \det {\rm Hess}(u) \circ IJ,
\end{split} 
\end{equation}
as claimed.  The proof for $\theta_I$ is similar.
\end{proof}

\begin{lemma}\label{toric-soliton}
Let $(F,g, I,J)$ be an $A$-deformation of a locally toric K\"ahler structure $(g_K, J_K)$, where ${\bf G}^K= {\rm Hess}(u)$.  Suppose further that $(g_K, J_K)$ is a KRS with constants $b \in \tor, c \in \tor^*$ guaranteed by Lemma \ref{l:KRS}.  Then the GK data satisfies
\begin{equation}\label{symplectic-GKRF} \rho_J + L_{JX_J} F = 0, \qquad \rho_I + L_{IX_I} F = 0, \end{equation}
where $X_J:= \frac{1}{2}\left( b + A(c)\right) \in \tor$ and $X_I := \frac{1}{2}\left(b - A(c)\right)\in \tor$.
\end{lemma}
\begin{proof} Using that $(I, J)$ are $F$-conjugate as in (\ref{eq:IJconj}) and that $F$ is closed, we have
\begin{equation} \label{JvJ}
 L_{JX_J} F= d(F(JX_J)) = d I F(X_J) = -dd^c_I \mu^{X_J},\end{equation}
where $\mu^{X_J}:= \langle \mu, X_J\rangle$ is the momenta of $X_J\in \tor$.  As $X_J= \frac{1}{2}\left(b + A(c)\right)$,  we compute
\begin{equation}\label{mu-vJ}
\begin{split}
 2\mu^{X_J} =&\ \langle b, \mu \rangle +  \langle A(c), \mu \rangle
= \langle b, \mu \rangle -  \langle c, A(\mu)\rangle 
= \langle b, \mu \rangle - \langle c, \nabla u \rangle  + \langle c, \nabla u - A(\mu)\rangle.
\end{split}\end{equation}
Now rearranging, applying $d d^c_I$, and using Lemmas~\ref{Bismut-Ricci} and \ref{pluriharmonic} we observe
\begin{align*}
    \rho_J + L_{J X_J} F =&\ \tfrac{1}{2} d d^c_I \left( \log \det \Hess(u) - 2 \mu^{X_J} \right)\\
    =&\ \tfrac{1}{2} d d^c_I \left( \log \det \Hess(u) - \langle b, \mu \rangle + \langle c, \nabla u \rangle - \langle c, \nabla u - A(\mu)\rangle\right)\\
    =&\ \tfrac{1}{2} d d^c_I \left( \log \det \Hess(u) - \langle b, \mu \rangle + \langle c, \nabla u \rangle \right)\\
    =&\ 0,
\end{align*}
where the final line uses the KRS equation from Lemma \ref{l:KRS}.
The claim for the complex structure $I$ is treated similarly.
\end{proof}

In the theory of generalized K\"ahler-Ricci solitons, the soliton vector fields $X_I$ and $X_J$ take a particular canonical shape relative to the Lee forms and the soliton potential $f$, which is derived from global considerations using the Perelman $\mathcal F$-functional (cf. \cite{SU2}).  As it turns out, in the toric setting it is possible to derive a local expression for the potential function, which is crucial in establishing the equivalence between generalized K\"ahler and K\"ahler-Ricci solitons.
\begin{lemma}\label{GKRS} Given the setup of Lemma \ref{toric-soliton}, one has
\[ X_J = \tfrac{1}{2} J \left( \theta_J^{\sharp} - \N f \right), \qquad X_I = \tfrac{1}{2} I \left( \theta_I^{\sharp} - \N f \right),\]
where
\begin{equation}\label{f} f : = - \log \det ({\rm Hess}(u) + A) - \langle c, \nabla u \rangle.\end{equation}
\end{lemma}
\begin{proof}
By the definition of $\mu^{X_J}$ and   using $F= -2g(I+J)^{-1}$ (this follows from  \eqref{eq:IJconj}), see e.g. \cite{ASU2}),  we have $X_J = -F^{-1}(d\mu^{X_J})= \tfrac{1}{2}(I+J)g^{-1}(d\mu^{X_J})$.  Observe from \eqref{mu-vJ} and the KRS equation that
\begin{align*}
    2 \mu^{X_J} = \log \det \Hess(u) + \IP{c, \N u - A \mu}.
\end{align*}
We introduce the notation
\begin{align*}
\langle c \cdot, \left({\rm Hess}(u) \pm A\right), d\mu \rangle:= \sum_{k,j=1}^n \left({\rm Hess}(u) \pm A\right)_{kj}c_j d\mu_k,
\end{align*}
and observe the basic fact that
\begin{align*}
    d \IP{c, \N u} = \IP{c, \Hess (u), d \mu}.
\end{align*}
Using this we compute by Lemma~\ref{Lee-forms}: 
\[ 
\begin{split}
2 g(JX_J) =&\ g(J(I+J))g^{-1}(d\mu^{X_J}) \\
=&\ -\tfrac{1}{2}d\log \det {\rm Hess}(u) + \tfrac{1}{2}d\log \det {\rm Hess}(u)\circ IJ \\
&\ -\tfrac{1}{2} d\langle c, \nabla u - A\mu\rangle + \tfrac{1}{2}d\langle c, \nabla u - A\mu\rangle \circ IJ\\ 
=&\ -\theta_J - d\log \det ({\rm Hess}(u)+A) \\
&\ -\tfrac{1}{2} d\langle c, \nabla u - A\mu\rangle + \tfrac{1}{2} \langle c, ({\rm Hess}(u)- A), d\mu \rangle \circ IJ  \\
=&\  -\theta_J - d\log \det ({\rm Hess}(u)+A) \\
&\ -\tfrac{1}{2} d\langle c, \nabla u - A\mu\rangle - \tfrac{1}{2}\langle c, ({\rm Hess}(u) +A), d\mu \rangle\\
=&\ -\theta_J - d\log \det ({\rm Hess}(u)+A) - d\langle c, \nabla u \rangle\\
=&\ - \theta_J + df,
\end{split}\]
using \eqref{boulanger} and \eqref{I-matrix} to express $d\mu \circ IJ$ in passing from the 3rd equality to the 4th.
Rearranging gives the claimed expression for $X_J$, and the computation for $X_I$ is similar, noting that $\det(\Hess(u) + A) = \det (\Hess(u)-A)$). \end{proof}

\begin{proof}[\bf Proof of Theorem \ref{t:Atypeconstruction}] If $(g_K, J_K)$ is a KRS, the claim for $(g, I, J)$ follows from Lemmas~\ref{toric-soliton} and \ref{GKRS} above and the characterization of GKRS (\cite{SU2} Proposition 4.1).

In the other direction, using (\cite{SU2} Proposition 4.1) and (\cite{ASU2} Lemma 2.9), for any symplectic-type GKRS the identity \eqref{symplectic-GKRF} holds. 
It follows using the general identity \eqref{JvJ} that the moment map $\mu^{X_J}$ of $X_J$ satisfies
\[ dd^c_I \mu^{X_J} = \rho_J.\]
By Lemma~\ref{Bismut-Ricci} we have that $2 \mu^{X_J}-\log \det {\rm Hess}(u)$ is a $\tor$-invariant $I$-pluriharmonic function, hence by Lemma~\ref{pluriharmonic} there exists $c_J \in \tor$ such that
\begin{gather} \label{thm1:1} 2\mu^{X_J}-\log \det {\rm Hess}(u)= \langle c_J, \nabla u- A \mu\rangle + const.
\end{gather}
We thus obtain the expression \eqref{mu-vJ} for $\mu^{X_J}$ with constant $c_J$, and hence the calculation of Lemma~\ref{GKRS} (which only uses the identities \eqref{symplectic-GKRF} and \eqref{mu-vJ} that we have just established) can be repeated to show that 
\[X_J = \tfrac{1}{2}\left( J\theta_J^{\sharp}  - J\nabla f_J\right),  \]
with $f_J$ given by \eqref{f}, i.e.
\[ f_J = -\log \det ({\rm Hess}(u) + A) - \langle c_J, \nabla u \rangle + const.\]
By a directly analogous argument we obtain that there exists $c_I \in \tor^*$ such that
\begin{gather} \label{thm1:2} 2 \mu^{X_I} - \log\det {\rm Hess}(u) = \langle c_I, \nabla u+A \mu \rangle + const,
\end{gather}
and 
\[X_I - I\theta_I^{\sharp} = - I\nabla f_I
\]
with 
\[ f_I = -\log \det ({\rm Hess}(u) - A) - \langle c_I, \nabla u \rangle + const.\]
As $(g, I, J)$ is a GKRS, there exists a smooth function $f$ such that $X_J- J\theta_J^{\sharp} = -J\nabla f$ and $X_I-I\theta_I^{\sharp} = - I \nabla f$. 
It follows from the expressions above 
\[ f_J = f + const, \qquad f_I = f + const.\]
Using further that $\det (\Hess(u) + A) = \det(\Hess(u) - A)$, it follows that $f_I - f_J = \langle c_J - c_I, \nabla u \rangle + const$ is a constant, and thus $c_I=c_J$, henceforth denoted $c$.  Letting 
\[ b:= X_I + X_J\in \tor,\]
equations (\ref{thm1:1}) and (\ref{thm1:2}) now yield
\[ \mu^b = \langle b, \mu\rangle= \log \det {\rm Hess}(u) + \langle c, \nabla u\rangle + const, \]
so $(g_K, J_K)$ is a toric KRS with potential function $f_K= -\langle b, \mu \rangle$ by Lemma~\ref{l:KRS}.
\end{proof}

\begin{rmk}
In dimension 4, if $\left<c,b\right>\neq 0$, then $X_I:=\tfrac{1}{2}(b-Ac)$ and $X_J:=\tfrac{1}{2}(b+Ac)$ give a basis of $\mathfrak{t}$.  Using this basis, we see the polytope in $\R_{x,y}$.
With respect to the dual basis $X_I^*,X_J^*$, it follows that $c=\gl (X_I^*+X_J^*)$ where $\gl=\frac12\la c,b\ra$.
The MA equation (where $\Hess\ u$ is computed in this basis) is
\begin{align*}
    \log\det\Hess\ u =& -\gl \la \nabla u, X_I^*+X_J^*\ra + \la \mu, X_I+X_J\ra +const\\
    =& -\gl(u_x + u_y) + x+y +const.
\end{align*}
We expound on this construction further in \S \ref{s:reduction}.
\end{rmk}

\begin{rmk}\label{r:GKRSpotential} It follows from the proofs above that the GK soliton potential $f$ can be explicitly computed
\begin{align*}
    f =&\ - \log \det (\Hess(u) + A) - \IP{c, \N u} + const\\
    =&\ \log \frac{\det \Hess(u)}{\det (\Hess(u) + A)} - \log \det \Hess(u) - \IP{c, \N u} + const\\
    =&\ \log \frac{\det \Hess(u)}{\det (\Hess(u) + A)} - \IP{b, \mu} + const.
\end{align*}
See also Proposition \ref{p:explicitf} below.
\end{rmk}

\begin{rmk} \label{r:horizontalRicci}
A general result of Kolesnikov (\cite{kolesnikov2012hessian} Theorem 1.1) (cf. also \cite{ustinovskiy2024geometry} Proposition 4.5) shows that the drift Monge-Amp\`ere equation we derive for GKRS implies that the horizontal metric has nonnegative Bakry-\'Emery-Ricci tensor with potential function $\IP{c, \N u} - \IP{b, \mu}$.
\end{rmk}

\subsection{Legendre duality for toric K\"ahler-Ricci solitons}\label{s:Legendre}

Here we record consequences of Legendre transform and their relationship to $A$-deformations.  To begin, for a convex function $u$ we define the Legendre transformation $\phi$ via
\begin{gather} \label{f:legendre}
\phi(y) + u(x) = \langle x, y \rangle, \qquad y:= \nabla u(x), \quad x = \nabla \phi(y),
\end{gather}
Assuming $u$ is the symplectic potential for a toric KRS, let
\[ g_K={\rm Hess}(u) \oplus {\rm Hess}(u)^{-1} \]
be the associated Riemannian metric defined on $U \oplus \tor$ where $x\in U\subset \tor^*, \, u(x) \in C^{\infty}(U)$, and ${\rm Hess}(u)$ is seen as a smooth section of $S^2\tor$ over $U$ whereas its inverse is a smooth section of $S^2\tor^*$ over $U$.  Similarly we define
\[ g_K^* = {\rm Hess}(\phi) \oplus   {\rm Hess}(\phi)^{-1}\]
on $V\oplus \tor^*$, where $V\subset \tor$ is the image of $U$ via $\nabla u : U \to V$.  The following lemma is an elementary consequence of the key property ${\rm Hess}(u)^{-1} ={\rm Hess}(\phi)$ of Legendre transformation:

\begin{lemma} If $u$ is the symplectic potential of a locally toric KRS $g_K$ satisfying (\ref{f:KRS}), then $g_K^*$ is a locally toric KRS with potential $\phi$ satisfying (\ref{f:KRS}), with the roles of $b$ and $c$ interchanged.
\end{lemma}

Now fix a \emph{nondegenerate} element $A\in \Lambda^{2}\tor$, which we see as a linear automorphism $A: \tor^* \to \tor$ with inverse $A^{-1} : \tor \to \tor^*$. 
\begin{lemma} The function $\phi_A(x):= \phi(Ax)$ is a convex function on $U_A := A^{-1}(V) \subset \tor^*$ which satisfies the KRS equation \eqref{f:KRS} with soliton constants $A(c) \in \tor$ and $A^{-1}(b) \in \tor^*$.
\end{lemma}
Denote by $\tilde g_K^*$ the KRS metric associated to $\phi_A$.  Note that this metric pulls back to $g_K^*$ by $(A^{-1}, A) : V  \oplus \tor^* \to U_A \oplus \tor$. Explicitly we have
\[ {\tilde g}^*_K = -A\left({\rm Hess}(\phi)\right)A \oplus -A^{-1}\left({\rm Hess}(\phi)\right)A^{-1}.\]
Further, consider the Riemannian metric $\tilde g^*$ on $U_A \oplus \tor$ corresponding to the $A$-deformation of ${\tilde g}^*_K$ associated to the bivector $(-A)=A^*$ (cf. \eqref{g}):
\[ 
\begin{split} {\tilde g}^* &= -A\left({\rm Hess}(\phi)\right)A \oplus -\left(\left(A\left({\rm Hess}(\phi)\right)A + A\right)^{-1}\right)^s  \\
&= -A\left({\rm Hess}(u)\right)^{-1}A \oplus -\left(\left(A\left({\rm Hess}(u)\right)^{-1}A + A\right)^{-1}\right)^s.
\end{split}\]
Alternatively, let 
\[ g= {\rm Hess}(u) \oplus \left(({\rm Hess}(u) + A)^{-1}\right)^s\]
be the Riemannian metric on $U\oplus \tor$, obtained via the $A$-deformation of $g_K$ and  
\[{\tilde g} := \left((\nabla \phi) \circ A\right)^* g = -A \left({\rm Hess}(\phi)\right) A \oplus \left(({\rm Hess}(u) + A)^{-1}\right)^s \]
its pullback to a metric on $U_A \oplus \tor$.
\begin{lemma}\label{legendre+A} Given the setup above, $\tilde g^* = \tilde g$.
\end{lemma}
\begin{proof}  Comparing the two metrics, we only need to show the identity
\begin{equation}\label{crazy} \left(\left({\rm Hess}(u) + A\right)^{-1}\right)^{s} = -\left(\left(A\left({\rm Hess}(u)^{-1}\right)A + A\right)^{-1}\right)^s,\end{equation}
which we can view as a general identity for a couple of a symmetric positive definite and antisymmetric matrices. 
Let ${\bf \Psi} ={\bf \Psi}^s + {\bf \Psi}^{a}$ be any nondegenerate matrix whose symmetric part ${\bf \Psi}^s$ is nondegenerate. By the first equality in \eqref{linear-algebra}, we have
\[\begin{split}
    \left(\left({\bf \Psi}^{-1}\right)^s\right)^{-1} &= {\bf \Psi}^{T}\left({\bf \Psi}^s\right)^{-1} {\bf \Psi}=({\bf \Psi}^s -{\bf \Psi}^{a})\left({\bf \Psi}^s\right)^{-1}({\bf \Psi}^s +{\bf \Psi}^{a})\\
    &= {\bf \Psi}^s - {\bf \Psi}^{a}({\bf \Psi}^s)^{-1}{\bf \Psi}^{a}. 
\end{split}\]
Letting ${\bf \Psi}= {\rm Hess}(u) + A$ and ${\bf \Psi}=-A({\rm Hess}(u))^{-1} - A$ in the above general identity, we get  
\[ 
\begin{split}  \left(\left(\left({\rm Hess}(u) + A\right)^{-1}\right)^{s}\right)^{-1} &= {\rm Hess}(u)- A({\rm Hess}(u))^{-1}A, \\
 \left(\left(\left(A({\rm Hess}(u))^{-1}A + A\right)^{-1}\right)^{s}\right)^{-1} &= A({\rm Hess}(u))^{-1}A -A(A^{-1}({\rm Hess}(u))A^{-1})A \\
 &= A({\rm Hess}(u))^{-1}A -{\rm Hess}(u).
\end{split} \]
The equality \eqref{crazy} follows from the above.
\end{proof}
The geometric meaning of Lemma~\ref{legendre+A} comes from the following:
\begin{prop}\label{GK-duality} Suppose $(g, I, J, F, \tor)$ is an $A$-deformation of a locally toric K\"ahler structure $(g_K, J_K, F, \tor)$ with $A \in \Lambda^2 \tor$ nondegenerate.  Then $(g, I, J)$ is symplectic-type with respect to two $\tor$-invariant symplectic forms, $F_+:=F$ and $F_- := F (I+J)(I-J)^{-1}$, with corresponding $\tor$-momenta $\mu^+$ and $\mu^-$ related by
\[ \mu^- = A^{-1} \nabla u + const, \qquad \mu^+ = A\nabla \phi + const, \]
where $u(x)$ is a symplectic potential for  $(g_K, J_K)$ and $\phi(y)$ is its Legendre transform.  Furthermore, $(g,-I, J, F_-)$ is an $A$-deformation of the toric K\"ahler structure $(g_K^-, J_K^-, F_-, \tor)$ with bi-vector $-A$ and symplectic potential $\phi(Ax)$.

In particular, if $(g, I, J)$ is a GKRS as above, then the corresponding K\"ahler metrics $(g_K^+, J_K^+, F_+, \tor)$ and $(g_K^{-}, J_K^{-}, F_-, \tor)$ are locally toric steady KRS, obtained from the metrics $g_K$ and $\tilde g_K^*$ defined above by pullback by the respective local coordinate maps $(t, \mu^+)$ and $(t, \mu^{-})$.
\end{prop}

\begin{proof} As $A$ is assumed nondegenerate, we obtain by \eqref{J-poisson} that the Poisson tensor $\gs$ is nondegenerate.  From the general theory of nondegenerate GK structures (see e.g. \cite{ASnondeg}), we know that the structure is symplectic-type with respect to the symplectic form
\[ F_- = -2g(J-I)^{-1}.\]
Using the expressions  \eqref{boulanger}, \eqref{g} and \eqref{I-matrix} we compute
\[ \begin{split}
&\left(\frac{J-I}{2}\right)^{-1} = 
\sum_{i,j=1}^n \left(-[({\bf \Psi}^{-1}]^{a}]^{-1})^{a}_{ij}d\mu_i \otimes \frac{\partial}{\partial t_j} +(A^{-1})_{ij} dt_i \otimes \frac{\partial}{\partial \mu_j}\right) \\
&- d\mu_i^{-}=F_-\left(\frac{\partial}{\partial t_i} \right)= -g\left(\frac{J-I}{2}\right)^{-1}\left(\frac{\partial}{\partial t_i}\right) =  -\sum_{k,j=1}^n A^{-1}_{ij}u_{, jk} d\mu_k= - d\left(\sum_{j=1}^nA^{-1}_{ij}u_{,j}\right).
\end{split}\]
It follows that up to an additive constant,  $\mu^{-} = A^{-1} \nabla u$. 

Substituting back in \eqref{boulanger}, \eqref{g} and \eqref{I-matrix}, we express $g$, $J$ and $I$ in the coordinate system $(t, \mu^{-})$ as follows
\[ 
\begin{split} 
g &= \sum_{i,j=1}^n-\left(A\left({\rm Hess}(u)\right)^{-1}A\right)_{ij} d\mu^-_i d\mu^-_j + \left(\left({\rm Hess}(u) + A\right)^{-1} \right)^s_{ij} dt_idt_j. \\
J &= \sum_{i,j=1}^n -\left(A + A\left(({\rm Hess}(u))^{-1}\right) A\right)_{ij}   d\mu^{-}_i \otimes \frac{\partial}{\partial t_j} - (A+ A\left(({\rm Hess}(u))^{-1}\right)A)^{-1}_{ij} dt_i \otimes \frac{\partial}{\partial \mu^{-}_j} \\
I &= \sum_{i,j=1}^n -\left(A- A\left(({\rm Hess}(u))^{-1}\right)A\right)_{ij}d\mu^{-}_i \otimes \frac{\partial}{\partial t_j}- \left(A- A\left(({\rm Hess}(u))^{-1}\right)A\right)^{-1}_{ij} dt_i \otimes \frac{\partial}{\partial \mu^{-}_j}.
\end{split}
\]
Let $\phi(y)$ be the Legendre transform of $u(x)$, and let $\phi_A(x):= \phi(Ax)$: using that ${\rm Hess}(\phi)= {\rm Hess}(u)^{-1}$, we have
\[ {\rm Hess}(\phi_A) = -A({\rm Hess}(\phi)) A = -A\left({\rm Hess}(u)^{-1}\right) A.\]
It thus follows that $(J, -I)$ are obtained from the toric $F_-$-compatible K\"ahler metric
\[ g_K^- = \sum_{i,j=1}^n \left(\phi_A(\mu^-)\right)_{,ij} d\mu^-_i d\mu^{-}_j + \left(\phi_A(\mu^-)\right)^{,ij} dt_i dt_j \]
by an $A$-deformation with the bivector $-A$.  The above also shows that $\phi_A(x)$ is the corresponding local symplectic potential of $(g, -I, J, F_-, \tor)$.  The formula relating $\mu^+$ and $\phi$ follows by duality of the arguments above.
\end{proof}

\section{Global Theory of toric generalized K\"ahler-Ricci solitons} \label{s:global}

In this section we turn to the global aspects of the relationship between toric KRS and GKRS.  We first address the question of when a toric GKRS can be extended from a dense open set of a given manifold, showing an equivalence to the well-known Abreu-Guillemin boundary conditions, and record some further geometric consequences in this setting.  Next we recall some elementary facts about length metric spaces and then give the proof of Theorem \ref{t:completeness}.  We then prove some properties of weighted scalar curvatures on complete GKRS, and illustrate the theory in the case of the Gibbons-Hawking ansatz.

\subsection{Extension theory} \label{ss:extension}

The first question to address is when K\"ahler and generalized K\"ahler structures related by an $A$-deformation can be extended from $\mathring{M}$ to $M$.  In the compact Dezlant case, the equivalence of the K\"ahler and generalized K\"ahler metrics extending is established for $n = 2$ in (\cite{boulanger2019toric} Theorems 10 and 11) and in (\cite{wang2022toric} Theorems 4.9 and 4.11) in general, based on a local argument in (\cite{apostolov2004hamiltonian} Lemma~2).  We drop here both the compact and Delzant hypotheses, and are abel to adapt the arguments of \cite{wang2022toric} to our local toric setup.

\begin{prop} \label{p:complete} Suppose $(F, g, I, J, \tor)$ is an $A$-deformation of a locally toric K\"ahler structure $(F, g_K, J_K, \tor)$ defined on $\mathring{M}$.  Suppose $(F, \tor)$ extends smoothly to a manifold $M$ such that $\mathring{M}$ is a dense open subset in $M$.  Then $(g_K, J_K)$ extends smoothly to $M$ if and only if $(g, I, J)$ does.  In this case, one has $g_K \geq g$.  In particular, if $(g, I, J)$ is defined on $M$ and $g$ is complete Riemannian metric on $M$, then $g_K$ is also a complete Riemannian metric on $M$.
\end{prop}

\begin{proof}
Suppose that $(g,I,J)$ is defined on $M$.  We use \eqref{extension}, which holds on $\mathring{M}$,  and notice that the expression
\[\sum_{i,j=1}^n \Big(A({\bf \Psi}^{-1})^s A\Big)_{ij}\left(\frac{I+J}{2}\right) d\mu_i\left(\frac{I+J}{2}\right)d\mu_j  \]
extends to $M$ because $({\bf \Psi}^{-1})_{ij}^s= g(X_i, X_j)$ is a matrix with smooth coefficients on $M$, $A$ is a constant matrix, and the 1-forms $d\mu_i = -F^{-1}(X_i)$ are defined on $M$ by our assumption.

If $(g_K, J_K)$ extends to $M$, we now use formula \eqref{extension-K} which holds on $\mathring{M}$. As $\left({\bf \Psi}^s\right)^{-1}_{ij} = g_K(X_i, X_j)$ is extendable, we conclude as before that the tensor $g-b$ is extendable as soon as the matrix  $\left({\rm Id} +\left({\bf \Psi}^{s}\right)^{-1}A\right)$ (which clearly extends to $M$) remains nondegenerate. Following the proof of (\cite{wang2022toric}, Theorem 4.11), this is shown to be true by the following observation (which holds on $\mathring{M}$):
\[ 
\begin{split}
    \det \left({\rm Id} +\left({\bf \Psi}^{s}\right)^{-1}A\right) &= \det\left(({\bf \Psi}^s)^{-\tfrac{1}{2}}\right) \det\left({\bf \Psi}^s + A\right) \det\left(({\bf \Psi}^s)^{-\tfrac{1}{2}}\right) \\
    &= \det\left({\rm Id} + ({\bf \Psi}^s)^{-\tfrac{1}{2}}A({\bf \Psi}^s)^{-\tfrac{1}{2}} \right) \geq 1,
    \end{split}\]
where for the last inequality it is used that $({\bf \Psi}^s)^{-\tfrac{1}{2}}A({\bf \Psi}^s)^{-\tfrac{1}{2}} $ is an anti-symmetric matrix. This inequality holds on $M$ by continuity.  This shows that the tensor $g-b$  extends to $M$. In particular, both the symmetric part $g$ and the skew part $b$ are extended to $M$. As $F$ is symplectic on $M$ and $-FJ= g - b$, we get a smooth extension of $J$, which must be an integrable almost complex structure on $M$ by continuity. The almost complex structure $I$ is the $F$-conjugate of $J$, so it extends too.

It remains to obtain that $g>0$, which follows by showing that $(I+J)^{-1}$ extends smoothly.  To this end, we compute from \eqref{momentum-angle}, (\ref{boulanger}) and \eqref{I-matrix},
 \[ \left(\frac{I+J}{2}\right)^{-1} + J_K = \sum_{i,j=1}^n \left((({\bf \Psi}^{-1})^{s})^{-1} - {\bf \Psi}^s\right)_{ij} d\mu_i \otimes X_j. \]
 This tensor extends smoothly provided the matrix $\left((({\bf \Psi}^{-1})^{s})^{-1} - {\bf \Psi}^s\right)_{ij}$ extends. Using \eqref{linear-algebra}, we have
 \[(({\bf \Psi}^{-1})^{s})^{-1}={\bf \Psi}^T({\bf \Psi}^s)^{-1} {\bf \Psi}. \]
Thus
\[
\begin{split}
(({\bf \Psi}^{-1})^{s})^{-1} - {\bf \Psi}^s &= {\bf \Psi}^T({\bf \Psi}^s)^{-1} {\bf \Psi} - {\bf \Psi}^s \\
&=({\bf \Psi}^{s}-A)({\bf \Psi}^s)^{-1}({\bf \Psi}^s + A) - {\bf \Psi}^s \\
&=({\rm Id} - A({\bf \Psi}^s)^{-1})({\bf \Psi}^s + A) -{\bf \Psi}^{s} \\
&=-A({\bf \Psi}^{s})^{-1}A.
\end{split}\]
The final line is smooth on $M$ as $({\bf \Psi}^s)^{-1}_{ij} = g_K(X_i, X_j)$ is smooth and $A=(A_{ij})$ is constant.  With this, the identity $F = - 2g(I+J)^{-1}$ shows that both $g$ and $(I+J)$ are everywhere nondegenerate over $M$, so $g>0$.

 Finally, the inequality $g_K \geq g$ follows by \eqref{extension}, noting that the final term is non-positive on $\mathring{M}$ (and hence on $M$ by continuity), as $\left(({\bf\Psi}^{-1})^s_{ij}\right)=\left(g(X_i, X_j)\right) \geq 0$.  The completeness statement follows directly by this inequality and the Hopf-Rinow theorem.
 \end{proof}

Having established this inequality, we turn now to the question of when either structure is extendable.  We assume that $\tor$ integrates to a compact real torus $\T$-action and recall a fundamental definition:

\begin{defn}\label{global-assumption} We say that a toric symplectic manifold $(M, F, \T)$ with moment map $\mu: M \to \Delta \subset \tor^*$, is \emph{\standard} if
\begin{enumerate}
\item $\mu : M \to \Delta$ is proper,
\item $\Delta$ is a closed convex simple polytope in $\tor^*$ with a finite number of vertices.
\end{enumerate}
\end{defn}

\begin{rmk}
As in the compact case, such symplectic manifolds are classified~\cite{KL} by the Delzant construction applied to $(\Delta, {\bf L})$, where $\Delta$ is a simple polytope admitting  a Delzant labeling  ${\bf L}=\{L_1, \ldots, L_d\}$, i.e. for each vertex $v\in \Delta$, the affine linear functions $\{L^{v}_1, \ldots, L^v_{n}\} \in {\bf L}$ vanishing at $v$ define a basis $\{dL_1^v, \ldots, dL^v_n\}$ of a lattice $\Lambda \subset \tor$.  In this case, we have an induced canonical K\"ahler metric $(g_c, J_c, F)$ on $M$,  obtained as the quotient of the flat K\"ahler structure on ${\mathbb C}^d$, and a $J_c$-holomorphic action of a complex torus $\T_{\C}\cong ({\C}^*)^n$, containing $\T$ as a real form, see e.g. \cite{guillemin1994kaehler, abreu2001kahler, apostolov-notes}.  We shall refer to $J_c$  as the \emph{standard} complex structure on $(M, F, \T)$. Notice that $J_c$ is equivariantly isomorphic to the algebraic complex structure on $M$,  associated to the fan determined by $(\Pol, {\bf L})$, see \cite{cifarelli2022uniqueness}.
\end{rmk}

In the Delzant setting,  we can give a complete characterization of the extendability of toric generalized K\"ahler structures in terms of the symplectic potential $u$, as in the Abreu-Guillemin theory.  Indeed, by Proposition~\ref{p:complete}, this is equivalent to the extendability of the corresponding K\"ahler structure. In the compact case, this is studied by M. Abreu in \cite{abreu2001kahler}, and in \cite{apostolov2004hamiltonian}, where the arguments are local around any point at the boundary. The latter approach  was adapted to the non-compact setup, assuming the momentum map $\mu$ is proper (cf. \cite{sena2021uniqueness}).

\begin{prop}\label{p:global}  Assuming $(F, g, I, J, \T)$ is \standard, the data is globally defined on $(M, F, \T)$ if and only if the symplectic potential $u(x)$ on $\mathring{\Pol}$ satisfies the \emph{Abreu--Guillemin boundary conditions}~\cite{abreu2001kahler}: 
\begin{enumerate}
\item $u(x)-u_c(x)$ is smooth on $\Pol$,
\item $\frac{\det\left({\rm Hess}(u)\right)}{\det\left({\rm Hess}(u_c) \right)}$ is smooth and positive on $\Pol$,
\end{enumerate}
where $u_c(x):= \tfrac{1}{2}\sum_{j=1}^{d} L_j(x) \log L_j(x)$ is the symplectic potential of the standard K\"ahler structure $(g_c, J_c)$ on $(M, F)$.
\end{prop}

\begin{rmk} Under the Delzant hypothesis, observe by \eqref{J-poisson} that $\gs_J$ vanishes on the pre-image of any vertex of $\Pol$,  as these are fixed points for the $\T$-action.  In particular $\gs_J$ is not invertible, and thus neither is $\gs$, and so $(g, I, J)$ cannot be nondegenerate at such a point.  In particular, if $M$ is $4m$-dimensional and $A\in \Lambda^2 \tor$ is a nondegenerate element, then   $(g, I, J)$ is a nondegenerate GK structure on $\mathring{M}$ such that the symplectic form $F=F_+=-2g(I+J)^{-1}$ extends to $M$ whereas the symplectic form $F_-= -2g(I-J)^{-1}$ on $\mathring{M}$ does not.  Notice in particular that the Legendre duality discussed in Proposition~\ref{GK-duality} does not work globally.
\end{rmk}

\begin{prop} (cf. \cite{apostolov2023hamiltonian} Lemma 2.12) \label{p:polytopecond} Let $(M, F, g, I, J, \T)$ be a toric symplectic-type GKRS such that $(M, F, \T)$ is \standard, and the data defined over $\mu^{-1}(\mathring{\Delta})$ satisfies the conditions of Theorem \ref{t:Atypeconstruction}.
Then $c \in \tor^*$ is the unique constant such that \begin{align} \label{eq: polytope faces condition}
    \IP{\N L_j, c} = 2
\end{align}
for any $L_j \in \bf{L}$.
\end{prop}
\begin{proof} By the Delzant hypothesis we know that there exists a symplectic potential $u$ for the structure satisfying the boundary conditions of Proposition \ref{p:global}.  By the proof of Theorem \ref{t:Atypeconstruction} we know that this symplectic potential satisfies a Monge-Amp\`ere equation
\begin{align*}
    \log \det \Hess(u) + \IP{c, \N u} = \IP{b, \mu} + const.
\end{align*}
The result follows by noting the boundary conditions and the fact that the right hand side is smooth as in (\cite{apostolov2023hamiltonian} Lemma 2.12).
\end{proof}

\begin{prop}\label{p:standard} Let $(g_K, J_K)$ be a toric K\"ahler structure defined on a Delzant symplectic manifold $(M, F, \T)$ with a symplectic potential $u$ defined on $\mathring{P}$, fix $A\in \wedge^2 \tor$, and let $(g, I, J)$ be $A$-deformation of $(g_K, J_K)$.  
\begin{enumerate}
    \item If $|(\nabla u)(x)| \to \infty$ as $|x| \to \infty$, then $J_K$ is equivariantly biholomorphic to the standard complex structure $J_c$ on  $(M, F, \T)$. 
    \item If $|(\nabla u)(x) + Ax|$ (resp. 
$|(\nabla u)(x) - Ax|) \to \infty$ as $|x|\to \infty$, then 
$J$ (resp. $I$)  is equivariantly biholomorphic to the standard complex structure $J_c$ on  $(M, F, \T)$.

\end{enumerate}
\end{prop}
\begin{proof} (1) We first show that 
\begin{equation}\label{C-star}
\nabla u : \mathring{P} \to \tor \,  \, \, \textrm{is a diffeomorphism.}
\end{equation}
This is standard and is established in the bounded polytope case for instance in (\cite{AAS}, Lemma A.1). We record the argument here, suitably adapted to the unbounded case, for the reader's convenience.  By the strict convexity of $u$, $x \to \nabla u (x)$ is a local diffeomorphism  from $\mathring{\Pol}$ to $\tor \cong \R^n$. It is injective, by using the convexity of $\mathring{\Pol}$ and that $\langle \nabla u(p(t)), \dot p(t) \rangle$ is strict monotone along a line segment $p(t)$ joining two points in $\mathring{\Pol}$. To establish the surjectivity, we  show that ${\rm Im}(\nabla u)$ is closed (it is already open and non-empty). To this end, if $q_k := \nabla u(p_k) \to q \in \R^n$, then we know that $p_k$ stay bounded by the assumption $\lim_{|x| \to \infty} |\nabla u(x)| = \infty$, so we can extract a subsequence $p_k \to p \in \Pol$. If $p \in \partial \Pol$, this will contradict the Abreu-Guillemin boundary condition (1) of Proposition~\ref{p:global},  which implies that $\lim_{x \to \partial \Pol} |\nabla u(x)| = \infty$. Thus, $p \in \mathring{\Pol}$ and thus ${\rm Im}(\nabla u)$ is closed. 

We next show that \eqref{C-star} implies that $J_K$ is standard.  One way to see this is to notice that \eqref{C-star} is in turn equivalent to the completeness of the vector fields $JX_v$ where $v \in \tor$.  This is because $\nabla u + \sqrt{-1} t$ defines $J_K$-holomorphic coordinates on $\mathring{M}$ such that the holomorphic vector fields $-J_K X_v + \sqrt{-1} X_v, \, v\in \tor$ become the standard vector fields in these coordinates.  The completeness of $JX_v$ is in turn equivalent to the extension of $\T$ to a holomorphic $(\C^*)^n$-action, in which case $J_K$ is standard by (\cite{cifarelli2022uniqueness} Lemma 2.14) (cf. also \cite{abreu2012scalar} Remark 2.11).  Alternatively, we can use \eqref{C-star} directly, as in \cite{apostolov2023hamiltonian} Remark 2.8, to build a $(\C^*)^n$-equivariant atlas on $(M, J_K)$ isomorphic to the one given by the fan defined by $\Pol$.

(2) One shows  similarly to the case (a)  that the map
\[ (\nabla u)(x) + A x : \mathring{\Pol} \to \R^n\]
is a diffeomorphism.  With this at hand, using Proposition~\ref{pluriharmonic}, we can define a $J$-holomorphic map $z:= e^{\left((\nabla u)(\mu) + A\mu + \sqrt{-1} t\right)}$ which, with respect to a suitable lattice basis,  defines an equivariant biholomorphism $\beta_v : (\mathring{M}, J) \cong (\C^{*})^n$ around each vertex of $\Pol$.  Thus by the proof of (\cite{ASU2} Lemma~3.6), the boundary conditions in Proposition~\ref{p:global} yield the extension of each $\beta_v$ to an equivariant $\C^n$-chart. The way these charts change varying the vertices recovers the canonical equivariant holomorphic atlas of $J_c$, obtained from the fan associated to $\Pol$. The argument for $I$ is similar. 
\end{proof}
An inspection of the explicit construction in \cite{apostolov2023hamiltonian} of Cao-type KRSs on $\C^n$ shows that the condition (2) of Proposition~\ref{p:standard} holds. Thus, Propositions~\ref{p:complete} and \ref{p:standard} above yield:
\begin{cor} $\R^{2n}$ admits complete symplectic-type rank 2 GKRSs $(g, I, J)$ such that $(\R^{2n}, I) \cong (\R^{2n}, J) \cong \C^n$.
\end{cor}

\subsection{Completeness}

We now turn to the main issue of this subsection, showing equivalence of completeness for K\"ahler metrics and their $A$-deformations under the Delzant hypothesis, providing a partial converse to Proposition \ref{p:complete}.  We will use the theory of path length spaces, and so briefly recall some aspects of this theory.  Let $(X, d)$ be a (path connected) metric space. For any continuous curve $\gamma(t): [0, 1] \to X$, its metric length is defined as 
\[ L_d(\gamma)= \sup\left\{ \sum_{i=1}^k d(\gamma(t_i), \gamma(t_{i+1})\right\},\] where the $\sup$ is taken over all finite partitions $0=t_0<t_1< \cdots <t_k=1$ of $[0,1]$.  Minimizing the metric length of curves yields a distance function $d^*$, and the metric space $(X, d)$ is a \emph{length metric space} if $d^* = d.$  This is equivalent to asking that for any two points $x, y \in X$ and any $\varepsilon>0$, there exists a continuous curve $\gamma(t): [0, 1] \to X$ joining $x$ and $y$ such that
\[ L_d(\gamma) < d(x, y) + \varepsilon.\]
We refer to \cite{petrunin2023pure} (see Ch.1, K and L) for further background.

\begin{thm} \label{t:completeness_text} (cf. Theorem \ref{t:completeness}) Suppose $(M, F, \T)$ is a non-compact toric symplectic manifold obtained by the Delzant construction with respect to a simple convex unbounded polytope $\Pol \subset \R^n$. Let $(g, I, J,f)$ and $(g_K, J_K)$ be respectively a $\T$-invariant GKRS and $\T$-invariant KRS on 
$M$ related by an $A$-deformation over $\mathring{M}$ as in Theorem~\ref{t:Atypeconstruction}.  Then $g$ is complete if and only if $g_K$ is complete.
\end{thm}

\begin{proof}  Let $N:= M/\T$ be the quotient space for the $\T$-action on $M$. Under the Delzant hypothesis, it is shown in \cite{KL} that $N$ is a manifold with corners isomorphic to $\Pol$. The Riemannian metrics $g$ and $g_K$ on $M$ define, respectively,  distance functions $d^M_g$ and $d^M_{g_K}$ on $M$; these induce Hausdorff distance functions on the space of orbits $N$, denoted respectively $d^N_g$  and $d^N_{g_K}$. 

We first observe that $(M, g)$ (resp. $(M, g_K)$) is complete if and only if $d^N_g(q_0, \cdot)$
(resp. $d^N_{g_K}(q_0, \cdot)$) is proper.  This follows because Riemannian manifolds are complete if and only if their induced distance functions are proper, and then observing that properness of the distance function on a metric space is equivalent to properness of the distance function on any quotient via an isometric action by a compact group.  Hence we aim to show that $d^N_g = d^N_{g_K}$.

Observe a general point that the property of being a length space is inherited by quotients of length spaces by compact isometric group actions.  In particular a quotient of any (potentially noncompact, noncomplete) Riemannian manifold by an isometric torus action is a length space.  Thus, consider $\pi : \mathring{M} \to \mathring{N}$, where $\mathring{M}$ is the open dense subset of points of $M$  which have $n$-dimensional $\T$-orbits and $\mathring{N} = \mathring{M}/\T$. Using the Delzant hypothesis, $\mathring{M} = \mu^{-1}(\mathring{\Pol})$ and all points in $\mathring{M}$ have trivial stabilizors. Thus, $\mathring{N}$ is a smooth manifold diffeomorphic to $\mathring{\Pol}$ through the induced quotient momentum map $\bar \mu : N \to \Pol$.  We suppress this identification and assume that $\mathring{N}= \mathring{\Pol}$.  The (incomplete) Riemannian manifolds  $(\mathring{M}, g)$ and $(\mathring{M}, g_K)$ induce Riemannian distance functions denoted $d^{\mathring{M}}_g$ and $d^{\mathring M}_{g_K}$ respectively.  By the above remark this yields induced length-space structures on $\mathring{N}$ denoted $d^{\mathring{N}}_g$ and $d^{\mathring{N}}_{g_K}$ respectively.  Finally, as $\mathring{\Pol}$ is contractible there exists a global symplectic potential $u$, well-defined up to the addition of an affine linear function on $\mathring{\Pol}$.  This then yields a well-defined Riemannian metric $g_{\rm red} := {\rm Hess}(u)$ on $\mathring{N}$, which induces a distance which we denote $d^{\mathring{N}}_{g_{\rm red}}$.  In summary, we have defined three a priori different distances on $\mathring{N}$: $d^{\mathring{N}}_g$, $d^{\mathring{N}}_{g_K}$ and $d^{\mathring{N}}_{g_{\rm red}}$.

The next point is to establish that in fact $d^{\mathring{N}}_g= d^{\mathring{N}}_{g_K}= d^{\mathring{N}}_{g_{\rm red}}$.  It is enough to establish the equality $d^{\mathring{N}}_g=d^{\mathring{N}}_{g_{\rm red}}$, as the argument for $d^{\mathring{N}}_{g_K}=d^{\mathring{N}}_{g_{\rm red}}$ is similar.  To this end, let $\tilde \gamma(t)\in \mathring{M}$ be a curve joining two points $p_1, p_2 \in \mathring{M}$ and $\gamma(t) = \pi \circ \tilde{\gamma}$ be the projected curve in $\mathring{N}$ joining $q_1=\pi(p_1)$ and $q_2=\pi(p_2)$. By the diagonal form \eqref{g} of $g$ in angular/momentum coordinates, we have
\begin{equation}\label{1} L^{\mathring{M}}_g(\tilde \gamma(t)) \geq L^{\mathring{N}}_{g_{\rm red}}(\gamma (t)). \end{equation}
The above shows that $d^{\mathring{N}}_g(q_1, q_2) \geq d^{\mathring{N}}_{g_{\rm red}}(q_1, q_2)$.  Indeed, for $q_1, q_2 \in \mathring{N}$ we can find $p_1, p_2 \in \mathring{M}$ such that $\pi(p_i)=q_i$ and $d^{\mathring{M}}_g(p_1, p_2) = d^{\mathring{N}}_g(q_1, q_2)$. Then, for any $\varepsilon >0$ there is a curve $\tilde\gamma(t) \in \mathring{M}$ joining $p_1$ and $p_2$ and projected curve $\gamma(t) = \pi(\tilde \gamma(t))$ joining $q_1$ and $q_2$ such that 
\[ \varepsilon + d^{\mathring{N}}_g(q_1, q_2) = \varepsilon + d^{\mathring{M}}_g(p_1, p_2) > L^{\mathring{M}}_g(\tilde \gamma) \geq L^{\mathring{N}}_{g_{\rm red}}(\gamma(t)) \geq d^{\mathring{N}}_{g_{\rm red}}(q_1, q_2).  \]
In the other direction, any curve $\gamma(t) \subset \mathring{N}$ between $q_1$ and $q_2$ can be lifted to a curve $\tilde \gamma(t) \subset \mathring{M}$ which in momentum angular coordinates has constant angular coordinates: we then have
\[L^{\mathring{M}}_g(\tilde \gamma(t)) = L^{\mathring{N}}_{g_{\rm red}}(\gamma (t)).\]
Observe that if $p_1, p_2$ are the end points of $\tilde \gamma(t)$, we have the inequalities 
\[ d^{\mathring{N}}_g(q_1, q_2) \leq d^{\mathring{M}}_g(p_1, p_2) \leq L^{\mathring{M}}_g(\tilde \gamma(t)) = L^{\mathring{N}}_{g_{\rm red}}(\gamma (t)).\]
As this holds for any curve $\gamma(t)$ joining $q_1, q_2$ we deduce $d_g^{\mathring{N}}(q_1,q_2) \leq d^{\mathring{N}}_{g_{\rm red}}(q_1, q_2)$.

With these preliminaries in place we are ready to establish a key lemma:
\begin{lemma}\label{l:local-metric} For any $q \in \mathring{N}$, there exists a neighbourhood $q\in U \subset \mathring{N}$ such that, for any $q_1, q_2 \in U$,
\[ d^{N}_g(q_1, q_2) = d^{N}_{g_K}(q_1, q_2) = d^{\mathring{N}}_{g_{\rm red}}(q_1, q_2).\]
\end{lemma}
\begin{proof} An elementary argument shows the following key property of $d^N_g$ (and $d^N_{g_K}$): for any $p_0 \in M$, the metric ball $B_{d^M_g}(p_0, r)$
of $(M, d^M_g)$ projects under $\pi$ to the metric ball $B_{d^{N}_g}(\pi(p_0), r)$ of $(N, d^N_g)$. Thus, if $p_0 \in \pi^{-1}(q_0) \in \mathring{M}$, using that $\mathring{M}$ is open, we can find a small geodesically convex ball $B_{d^M_g}(p_0, r) \subset \mathring{M}$. It will project to the ball $B_{d^{N}_g}(q_0, r)$ which is thus a subset on $\mathring{N}$.  We first claim that $B_{d^{\mathring{N}}_g}(q_0, r) = B_{d^{N}_g}(q_0, r)$ and $d^{\mathring{N}}_g= d^{N}_g$ on $B_{d^{N}_g}(q_0, \frac{r}{2})$, which  will yield the lemma as we have established $d_g^{\mathring{N}} = d_{g_{\rm red}}^{\mathring N}$ above.  Clearly, by definition, $d^{\mathring{N}}_g(q_0, q) \geq d^{N}_g(q_0, q)$. 
If $d_g^N(q_0, q)< r$, since $g$ is $\T$-invariant there exists a point $p \in \pi^{-1}(q)$ such that $d^M_g(p_0, p) = d^N_g(q_0, q)<r$, i.e. $p\in B_{d^M_g}(p_0, r)$.  The length of the geodesic segment $\tilde \gamma(t) \subset  B_{d^M_g}(p_0, r)$ joining $p_0$ and $p$  realizes $d^M_g(p_0, p)=d^N_g(q_0, q)$.  Hence it follows from \eqref{1} that $\tilde \gamma(t)$ projects to  a curve in $\mathring{N}$ such that $L^{\mathring{N}}_{g_{\rm red}}(\gamma(t))= d^{\mathring{N}}_{g_{\rm red}}(q_0, q)$.  As we established $g^{\mathring{N}}_{g_{\rm red}} = d_g^{\mathring{N}}$ above, it follows that $d^N_g(q_0, q)= d^{\mathring{N}}_{g_{\rm red}}(q_0, q)=d^{\mathring{N}}_g(q_0, q)$.   We can perform a similar argument for any two points $q_1, q_2 \in B_{d^N_g}(q_0, \frac{r}{2})$, using that $B_{d^M_g}(p_0, r)$ is geodesically convex.
\end{proof}

Finally we are ready to establish the key equality $d^N_g= d^N_{g_K}$. 
The point is to establish an analogue of Lemma~\ref{l:local-metric} on the preimage $\mathring{M}_{Q} \subset M$ of the interior of each face ${Q} \subset \Pol$.  This follows by the following facts: 
\begin{itemize}
    \item As $(M, F, \T)$ is Delzant, $M_Q=\mu^{-1}(Q)$ is a (smooth) toric submanifold of $(M, F, \T)$ with respect to torus action of $\check{\T}_Q=\T/\T_{Q}$,  where $\T_{Q} \subset \T$ is the common stabilizer in $\T$ of all points in $M_{Q}$.
    \item As $M_Q$ is the set of fixed points of $\T_Q$, which acts isometrically on $(M, g)$, the submanifold $(M_F, g)$ is totally geodesic.  The same statement holds for $(M_F, g_K)$.
    \item The data $(g_K, F)$ restricted to $M_Q$ defines a $\check{\T}_Q$-toric K\"ahler metric with  symplectic potential $u_{|_{\mathring{Q}}}$ defined on $\mathring{Q}$ (cf. proof of \cite[Prop.2]{donaldson2004interior}).  
    Using \eqref{extension}, $g$ restricted to $\mathring{M}_Q$ has the diagonal form \eqref{g} with potential function $u_{|_{\mathring{Q}}}$ and bivector $A_Q$ which is the projection of $A \in \Lambda^2\tor$ to $\Lambda^2(\tor/\tor_Q)$. In particular, $g$ and $g_K$ have equal horizontal parts on $\mathring{M}_Q$.
\end{itemize}
We thus obtain a stratification of $N$ with open metric subspaces $\mathring{N}_Q$, parametrized by the set of faces $Q \subset \Pol$.  On each open strata $\mathring{N}_Q$, we showed that locally $d_g^{\mathring{N}_Q}= d_{g_K}^{\mathring{N}_Q}$.  We conclude from here that the $d_g^N$-length of any curve in $N$ equals its $d^N_{g_K}$-length, and hence $d^N_g = d^N_{g_K}$.\end{proof}

\subsection{Global geometry of GKRS}

A standard fact of generalized Ricci solitons is that they have constant weighted scalar curvature \cite{OSW} (cf. \cite{GRFbook}).  In general the soliton potential function $f$ is obtained implicitly via the $\lambda$ functional, and the value of the weighted scalar curvature is also implicit, equal to the $\gl$ functional itself.  In the toric setting we obtain explicit expressions for both in terms of the canonically associated vectors $b$ and $c$.

\begin{lemma} \label{l:KRSscalar} Given $(M, F, g_K, J_K, f_K)$ a K\"ahler-Ricci soliton as in Lemma \ref{l:KRS}, the weighted scalar curvature satisfies
\begin{align} \label{f:CSC}
R_{g_K} + 2 \gD f_K - |\N f_K|^2 = - \IP{b,c}.
\end{align}
In particular,
\begin{align} \label{f:bcinequality}
\IP{b,c} \geq 0.
\end{align}
\end{lemma}
\begin{proof} We first recall that for a toric KRS the potential function $f_K$ is the pullback of the linear function $- \IP{b,x}$ by the momentum map $\mu$ (cf. (\cite{apostolov-notes} Lemma 3.1, \cite{cifarelli2022uniqueness} \S 2.3).  We thus get $X_b= -J_K\nabla f_K$ where $X_b= \sum_{i=1}^n b_i X_i $ is the fundamental vector field of $b$. We will actually compute that
\begin{align} \label{f:scalID}
    R_{g_K} + |X_b|^2 = R_{g_K}+ | \nabla f_K|^2 = \IP{b,c}.
\end{align}
Given this and the fact that $\Rc_{g_K} + \N^2 f_K = 0$, equation (\ref{f:CSC}) follows.  Furthermore, given that complete steady Ricci solitons have nonnegative scalar curvature \cite{ChenZhu}, (\ref{f:bcinequality}) follows.

To establish (\ref{f:scalID}), we use the Legendre transform $\phi(y)$ of the convex function function $u(x)$, defined  as in (\ref{f:legendre}) by
\[ \phi(y) + u(x) = \langle x, y \rangle, \qquad y:= \nabla u(x), \quad x = \nabla \phi(y),\]
to equivalently rewrite the KRS condition as
\begin{equation}\label{KRS-dual}
    - \log \det {\rm Hess}(\phi) + \langle c, y \rangle=  \langle b, \nabla \phi \rangle + const. \end{equation}
Using that $y=\nabla u$ are pluriharmonic, we compute
\[ 
\begin{split}
    \rho_{g_K} &=-\tfrac{1}{2} dd^c_{J_K} \log \det {\rm Hess}(\phi) = \tfrac{1}{2} dd^c_{J_K} \left(\sum_{j=1}^n b_j \phi_{, j}(y)\right) \\ 
    &= \tfrac{1}{2} d\left(\sum_{j,k=1}^n b_j \phi_{, jk}(y) dt_k\right) \\
    &= \tfrac{1}{2} \sum_{j, k, r=1}^n b_j\phi_{, jkr} dy_r \wedge dt_k = \tfrac{1}{2} \sum_{j=1}^n b_j{\rm Hess}(\phi)_{rk, j} dy_r \wedge dt_k. 
    \end{split}\]
The symplectic form $F$ is given in the $(t, y)$ coordinates by (see e.g. \cite{apostolov-notes})
\[ F = dd^c_{J_K} \phi(y)= \sum_{r,k=1}^n dy_r\wedge dt_k.\]
We use the above computations to express the scalar curvature
\[ 
\begin{split} 
R_{g_K} &= 2(\rho_{g_K} \wedge F^{[n-1]})/F^{[n]} = L_{Y_b} \left(\log \det {\rm Hess}(\phi) \right), \\
        \end{split} 
\]
where $Y_b = -J_KX_b= \sum_{j=1}^n b_j \partial_{y_j}$.  Using \eqref{KRS-dual} again, we have
\[ 
\begin{split} R_{g_K} =L_{Y_b} \left(\log \det {\rm Hess}(\phi) \right) &= \langle b,c \rangle - \sum_{i,j=1}^n ({\rm Hess}(\phi))_{ij} b_i b_j \\ 
& = \langle b,c \rangle - \sum_{i,j=1}^n \left({\rm Hess}(u)^{-1}\right)_{ij} b_i b_j \\ &= \langle b,c \rangle - g_K(X_b, X_b),
\end{split}\]
which yields the identity.
\end{proof}

\begin{defn} Given $(M, g, I, J)$ a generalized K\"ahler structure, and $\varphi \in C^{\infty}(M)$, the \emph{generalized scalar curvature} is
\begin{align*}
    R_g^{H, \varphi} = R_g - \tfrac{1}{12} \brs{H}^2 + 2 \gD \varphi - \brs{\N \varphi}^2,
\end{align*}
where $H= - d^c_J \omega_J = d^c_I \omega_I$ is the torsion.
\end{defn}

\begin{rmk} \label{r:scalarhistory} The first appearance of a natural scalar curvature in generalized K\"ahler geometry was in \cite{boulanger2019toric}, where a moment map for toric symplectic-type GK structures was derived, and thus taken to be the definition of scalar curvature, inspired by \cite{donaldson2003remarks}.  This was built upon by Goto in \cite{Goto_2020,goto-21}, first extending Boulanger's definition beyond the toric setting, then giving a general definition depending on a choice of volume form $\mu$.  As Goto's definition uses the local spinorial description of GK structures, it was unclear how to express the scalar curvature classically.  In the toric case Goto's spinorial description was shown to equal the Boulanger formulation in \cite{wang2020toric}.  Finally it was established in \cite[Theorem 1.2]{ASUScal} that this scalar curvature equals a constant multiple of $R^{H,\varphi}$ where $\mu = e^{-\varphi} dV_g$.  In the context of symplectic-type GK structures, Goto and Boulanger implicitly choose the volume form $\mu$ to be the symplectic volume form $\tfrac{1}{n!} F^n$, which then determines the weight function
\begin{align*}
    \Psi = - \log \frac{F^n}{dV_g},
\end{align*}
which agrees with one of the partial Ricci potential functions of Definition \ref{d:Riccipot} (cf. \cite{apostolov2022generalized, ASUScal}).
\end{rmk}

\begin{cor} \label{c:GKRSscalar} Let $(M, F, g, I, J, f)$ be a complete toric generalized K\"ahler-Ricci soliton which is an $A$-deformation of a toric steady K\"ahler-Ricci soliton $(g_K, J_K)$ over a dense open subset $\mathring{M} \subset M$. Then the generalized scalar curvature $R^{H,\Psi}$ of $(g, I, J)$ satisfies the pointwise inequality
\[ R_g^{H,\Psi} \geq 0.\]
If equality holds at one point, then $R^{H,\Psi} \equiv 0$ everywhere and $(F, g, I, J)$ is an $A$-deformation of a toric Ricci-flat K\"ahler metric $(g_K, J_K)$.
\end{cor}

\begin{proof} By the discussion of Remark \ref{r:scalarhistory}, the generalized scalar curvature of $(F, g, I, J)$ as defined in \cite{boulanger2019toric} is
\[R_g^{H,\Psi} = -\sum_{i,j=1}^n {\bf H}^K_{ij,ij},\]
which is also the scalar curvature $R_{g_K}$ of $(g_K, J_K)$ (see \cite{abreu2001kahler}). By Proposition~\ref{p:complete}, $(g_K, J_K)$ is a complete gradient steady K\"ahler-Ricci soliton.  Using (\cite{chow2023Ricci} Theorem 2.14, cf. \cite{chen2009strong}) we know that $R_{g_K} \geq 0$ and if $R_{g_K}=0$ at some point, then $g_K$ is Ricci-flat.
\end{proof}

\begin{cor} \label{c:rank1fromCY} Let $(g, I, J, F, \tor)$ be a complete toric GKRS on $M$, which is an $A$-deformation of a locally toric KRS $(g_K, J_K)$ on a dense open subset $\mathring{M} \subset M$, with $A\in \Lambda^2\tor$ nondegenerate.  Then $(g,I,J)$ has rank at most one if and only if $(g_K, J_K)$ is Ricci-flat. 
\end{cor}
\begin{proof} By Theorem~\ref{t:Atypeconstruction} and Lemma~\ref{toric-soliton}, the soliton Killing vector fields $X_I, X_J$ of $(g, I, J)$ are related to the soliton Killing vector field $b$ of $(g_K, J_K)$ by
\[ X_J= \tfrac{1}{2} \left(b+ A(c)\right), \qquad X_I = \tfrac{1}{2} \left(b - A(c)\right).\]
If $(g_K, J_K)$ is Ricci-flat, i.e. $b=0$, then $X_I = -X_J$ and thus the rank is at most one.  Conversely, if $0=X_I \wedge X_J=2A(c) \wedge b=0$, we have that $A(c)$ and $b$ are aligned. If $A(c)=0$, we have $c=0$ by the nondegeneracy assumption for $A$, and hence $\langle b,c \rangle =0$. Otherwise, we can write $b=\lambda A(c)$, and again $\langle b, c \rangle =0$ as $A$ is skew.  It follows by Lemma~\ref{l:KRSscalar} (cf. \eqref{f:scalID}) that 
$R_{g_K} + |X_b|^2=0$.  Since $g$ is assumed complete, $g_K$ is complete by Proposition~\ref{p:complete}, hence by (\cite{chow2023Ricci} Theorem 2.14, cf. \cite{chen2009strong}), $R_{g_K} \geq 0$.  We thus conclude that $X_b=0$, i.e. $g_K$ is a Ricci-flat K\"ahler metric. \end{proof} 

\begin{prop}\label{p:constantscal} Let $(M, g, I, J, f)$ be a GKRS which is an $A$-deformation of a locally toric KRS $(M, g_K, J_K, f_K)$.  Then
\begin{align*}
    R_g^{H,f} = - \IP{b,c} = R^{0,f_K}_{g_K}.
\end{align*}
\end{prop}
\begin{proof} As mentioned above, a GKRS satisfies, $R^{H,f}_g = const$ by \cite{ASUScal}.  For a toric KRS, we have $R^{0, f_K}_{g_K}= -\IP{b, c}$ by Lemma~\ref{l:KRSscalar}.  We will show $R^{H,f}_g=R^{0, f_K}_{g_K}$.  First, by \eqref{omega} we have
\begin{equation}\label{Psi-toric}
     \Psi =-\log(dV_F/dV_g) = -\log\left(\frac{\det({\rm Hess}(u) + A)}{\det({\rm Hess}(u))}\right),
\end{equation}  
so that, by Remark~\ref{r:GKRSpotential},
\begin{equation}\label{potentials} f = \Psi + f_K + const.\end{equation}
By  the proof of Corollary~\ref{c:GKRSscalar} we have
\[
\begin{split}R^{0,f_K}_{g_K}&=R_{g_k} + 2\Delta_{g_K} f_K -|df_K|_{g_K}^2 \\
&=-\sum_{i,j=1}^n{\bf H}^K_{ij,ij}+2\Delta_{g_K} f_K -|df_K|_{g_K}^2 \\
&= R^{H, \Psi}_g +2\Delta_{g_K} f_K -|df_K|_{g_K}^2 \\
&= R^{H, f}_g  + 2\Delta_g(\Psi-f) + |df|^2_g -|d\Psi|^2_g + 2\Delta_{g_K} f_K -|df_K|_{g_K}^2 \\
&= R_g^{H,f} - 2 (\gD_g - \gD_{g_K}) f_K + 2 \IP{d \Psi, d f_K}_g,
\end{split} \]
where for the last line we used \eqref{potentials} and $|df_K|_{g_K} = |df_K|_g$, see \eqref{g_K} and \eqref{g}.

We now show that for any $\tor$-invariant smooth function $\varphi$, the following equality holds
\begin{equation}\label{laplacians} \Delta_g \varphi -\Delta_{g_K}\varphi  = \IP{d\Psi, d\varphi}_{g}.\end{equation}
Applied to $f_K$,  we thus get $R^{H,f}_g = R^{0,f_K}_{g_K}=-\IP{b,c}$ as wanted.  Using momentum/angular coordinates, we first observe 
\begin{align*}
\Delta_{g_K}\varphi &= -\IP{dJ_K d\varphi, F^{-1}}  \\
&= \IP{\sum_{i,j=1}^nd(\varphi_{, i} {\rm Hess}(u)^{-1}_{ji}dt_j), \sum_{j=1}^n\frac{\partial}{\partial \mu_j} \wedge \frac{\partial}{\partial t_j} }\\
&=\sum_{i, j=1}^n \left(\varphi_{,ij}({\rm Hess}(u))^{-1}_{ij}+ \varphi_{,i}({\rm Hess}(u))^{-1}_{ji,j}\right).
\end{align*}
Next we observe
\begin{align*}
\Delta_g \varphi  - \langle d\varphi, \theta_J \rangle_g&= -\IP{dJd\varphi, \omega_J^{-1}} \\
&= \IP{\sum_{i,j=1}^n d(\varphi_{,i} {\bf \Psi}^{-1}_{ji}dt_j), \sum_{r,s=1}^n \left(({\rm Hess}(u))^{-1} {\bf \Psi}\right)_{rs}\frac{\partial}{\partial \mu_r} \wedge \frac{\partial}{\partial t_s}  } \\
&=\sum_{i, j=1}^n \varphi_{,ij}({\rm Hess}(u))^{-1}_{ij} + \sum_{i,r,k, s=1}^n \varphi_{, i}({\rm Hess}(u))^{-1}_{rk} {\bf \Psi}_{ks}{\bf \Psi}^{-1}_{si, r} \\
&=\sum_{i, j=1}^n \varphi_{,ij}({\rm Hess}(u))^{-1}_{ij} + \sum_{i,r,k=1}^n \varphi_{, i}({\rm Hess}(u))^{-1}_{rk,r} \left({\rm Hess}(u){\bf \Psi}^{-1}\right)_{ki} \\
&= \sum_{i, j=1}^n \left(\varphi_{,ij}({\rm Hess}(u))^{-1}_{ij} + \varphi_{,i}({\rm Hess}(u))^{-1}_{ji,j}\right) +\sum_{i,r=1}^n \left(({\bf \Psi}^T)^{-1}A({\rm Hess}(u))^{-1}_{,r} \right)_{ir} \varphi_{, i} \\
&= \Delta_{g_K}\varphi +\sum_{i,r=1}^n\left(({\bf \Psi}^T)^{-1}A({\bf \Psi}^s)^{-1}_{, r}\right)_{ir}\varphi_{,i}.
\end{align*}
By the fourth line of \eqref{long-theta}, and using \eqref{g} and \eqref{Psi-toric}, we have
\[
\begin{split}
\IP{d\varphi, \theta_J}_g &= \IP{d\varphi, d\Psi}_g + \sum_{j,r=1}^n \left(({\bf \Psi}^{s})^{-1}A({\bf \Psi}^T)^{-1} {\bf \Psi}^{s}_{, r}({\bf \Psi}^s)^{-1} \right)_{jr}\varphi_{, j} \\
&= \IP{d\varphi, d\Psi}_g - \sum_{j,r=1}^n \left(({\bf \Psi}^{s})^{-1}A({\bf \Psi}^T)^{-1} {\bf \Psi}^{s}({\bf \Psi}^s)^{-1}_{,r} \right)_{jr}\varphi_{, j}. 
\end{split} \]
Notice that for any nondegenerate matrix ${\bf \Psi}= {\bf \Psi}^s + A$ we have
\[({\bf \Psi}^s)^{-1} A ({\bf \Psi}^T)^{-1}{\bf \Psi}^s = ({\bf \Psi}^s)^{-1}({\bf \Psi}^s- {\bf \Psi}^T) ({\bf \Psi}^T)^{-1}{\bf \Psi}^s=({\bf \Psi}^T)^{-1} A.\]
We thus get
\[ \IP{d\varphi, \theta_J}_g = \IP{d\varphi, d\Psi}_g - \sum_{i,r=1}^n\left(({\bf \Psi}^T)^{-1} A ({\bf \Psi}^s)^{-1}_{, r}\right)_{ir} \varphi_{,i}.\]
Combining the above computations yields \eqref{laplacians} and hence the proposition.
\end{proof}
\begin{rmk}
Exponentiating \eqref{potentials}, we get 
\[ {\rm vol}:=e^{-f} dV_g = \lambda e^{-f_K} dV_F, \qquad \lambda \in \R_{>0}.\]
This provides a formal conceptual explanation for the conclusion of Proposition~\ref{p:constantscal}, via the infinite dimensional GIT setup in \cite{goto-21, ASUScal}.  Indeed, the above equality shows that $(g_K, J_K, f_K)$ and $(g, I, J, f)$ are zeroes of the momentum map on the space of generalized complex structures ${\mathbb I}$ compatible with the generalized complex structure $\JJ_F$ of $F$ and the volume form  $ {\rm vol}$. In this setup, $\IP{b, c}$ is a topological constant such that the normalized generalized K\"ahler scalar curvature
\[ {\mathbb I} \to \Gscal({\mathbb I}, \JJ, {\rm vol}) + \IP{b,c}. \]
is the corresponding momentum map, viewed via the $L^2(M, {\rm vol})$-product as a linear map on the space $C^{\infty}(M)/\R$ of Hamiltonian automorphisms with respect to $\JJ_F$.
\end{rmk}

\subsection{Toric rank \texorpdfstring{$1$}{1} GKRS}
We illustrate the theory above by constructing examples using the Gibbons-Hawking ansatz.  Fix $(M^4, g, I, J)$ a complete toric GKRS which is rank $1$, in the sense of Definition \ref{d:rank}.  Such structures are described by the Gibbons-Hawking ansatz as in (\cite{SU2} Theorems 1.2 and 1.3).  The hypothesis of an extra circle action on the space strongly restricts the possibilities that occur.  We describe the relevant spaces in the next lemma, assuming detailed familiarity with \cite{SU2}, which is also reviewed in the next section.

\begin{prop} \label{p:toricrank1class} A regular toric rank $1$ GKRS $(M^4, g, I, J)$ with Ricci curvature bounded below and finite topology is isometric to:
\begin{enumerate}
    \item A space described in \cite{SU2} Theorem 1.2(1).  In this setting $N(a_+, 0)$ is canonically diffeomorphic to $\mathbb C \times \mathbb R$, and $\{z_i\} \in \{0\} \times \mathbb R$.  If $\{z_i\} \neq \emptyset$ then $|k_+| = 1, l_+ = 0$.
    \item A space described in \cite{SU2} Theorem 1.2(2).
    \item A space described in \cite{SU2} Theorem 1.2(3).  In this setting $N(a_+,a_-)$ is canonically diffeomorphic to $S^2 \times \mathbb R$, and  moreover $\{z_i\} \in \{N,S\} \times \mathbb R$.  If $\{z_i\} \cap \{N\} \times \mathbb R \neq \emptyset$ then $\brs{k_+} = 1, l_+ = 0$.  If $\{z_i\} \cap \{S\} \times \mathbb R \neq \emptyset$ then $\brs{k_-} = 1, l_- = 0$.
    \end{enumerate}
    
    \begin{proof} First fix $(M^4, g, I, J)$ a complete toric rank $1$ GKRS.  In particular, we can describe the metric using the Gibbons-Hawking ansatz, and there exists a nontrivial biholomorphic Killing field $Z$ which commutes with $X_I$ and $X_J$ and can be chosen without loss of generality to be horizontal.  From the construction it is clear then that $Z$ preserves the reduced data $W, \eta$, and $h$.
    
    We now analyze the three cases of \cite{SU2} Theorem 1.2 in turn.  In case (1) we refer to \cite{SU2} \S 5.4.1.  In particular the moment image is canonically identified with $\mathbb C \times \mathbb R$, with the metric described in polar coordinates
    \begin{align*}
        (\mu_1, \rho, \mu_-) \in S^1 \times \mathbb R_{>0} \times \mathbb R, \qquad \rho = \exp(a_+ \mu_+/2)
    \end{align*}
    by
    \begin{align*}
        h_{a_+, 0} = \frac{4}{1 + \rho^2} \left( \frac{\rho^2}{1 + \rho^2} d \mu_1^2 + \frac{4}{a_+^2} d \rho^2 + 4 d \mu_-^2 \right).
    \end{align*}
    The metric completion extends this metric across the line $\{\rho = 0\}$ to have a cone singularity of angle $\pi |a_+|$.  The only vector fields preserving $h_{a_+,0}$ are the canonical rotation in the $\mathbb C$ factor and translation in the $\mu_-$ factor.  Now, as explained in \cite{SU2} \S 6.1, we can express $W = V \til{W}$, where $V$ is a positive harmonic function for an explicit metric $\til{h}$ conformally related to $h_{a_+,0}$, and $\til{W}$ is a certain canonical function (cf. \cite{SU2} Lemma 4.10).  As shown in \cite{SU2} Proposition 6.1, such functions are identified with a sum of Greens functions and possibly a positive constant.  Noting that the metric $\til{h}$ has the same symmetries as $h_{a_+,0}$, it follows that $W$ can only be invariant if the corresponding Greens functions have poles on the axis $\{\rho = 0\}$.  However, as explained in \cite{SU2} Proposition 6.6, the poles must lie in the smooth locus of $N(a_+, 0)$.  Thus if there are Greens function factors we force $|k_+|= 1, l_+ = 0$, as claimed.

    To establish case (2), it suffices to observe that all examples described in \cite{SU2} Theorem 1.2(2) do indeed admit further Killing fields as described above, i.e. rotation in the $\mathbb C$ factor and translation in the $\mu_-$ factor.

    For case (3), we refer to \S 5.4.2 of \cite{SU2}, noting that in this case the moment map image is defined using coordinates
    \begin{align*}
        ( \mu_1, \rho_1, \rho_2) \in S^1 \times \mathbb R_{> 0} \times \mathbb R > 0,
    \end{align*}
    where
    \begin{align*}
        \rho_1 =&\ \frac{e^{-a_-\mu_-/2}}{a_-^2 e^{a_+ \mu_+} + a_+^2 e^{-a_- \mu_-}}, \qquad
        \rho_2 = \frac{e^{a_+ \mu_+/2}}{a_-^2 e^{a_+ \mu_+} + a_+^2 e^{-a_- \mu_-}},
    \end{align*}
    by
    \begin{align*}
        h_{a_+,a_-} = \frac{4}{\rho_1^2 + \rho_2^2} \left( \frac{\rho_1 \rho_2}{\rho_1^2 + \rho_2^2} d \mu_1^2 + \frac{4}{a_-^2} d \rho_1^2 + \frac{4}{a_+^2} d \rho_2^2 \right).
    \end{align*}
    As explained in \cite{SU2} Proposition 5.13 the metric completion is isomorphic to $S^2 \times \mathbb R$ with the metric extending  with cone singularities of angles $\pi |a_-|$ and $\pi |a_+|$ across the subsets $\{\rho_1 = 0\}$ and $\{\rho_2 = 0\}$, which in turn correspond to the degeneracy loci of $\gs$.  This metric admits a Killing field given by rotation in the $\mu_1$ direction.  Arguing as in case (1) using \cite{SU2} Proposition 6.4, the possible functions $W$ are equivalent to positive harmonic functions $V$ for an explicit metric conformally related to $h_{a_+,a_-}$.  Such functions are shown to be sums of Greens functions, a positive constant, and a multiple of an anomalous solution to the Laplace equation in this setting.  The anomalous solution is invariant under the $\mu_1$-rotation (cf. \cite{SU2} Proposition 6.4 (2)).  Moreover, the Greens function terms are invariant under $\mu_1$-rotation if and only if the points lie on the axes $\{\rho_1=  0\} = \{N\} \times \mathbb R, \{\rho_2 = 0\} = \{S\} \times \mathbb R$. Again we know by \cite{SU2} Proposition 6.6 that the poles of the Greens functions must lie in the smooth locus of $N(a_+,a_-)$.  As in case (1) this forces $\brs{k_{\pm}} = 1$ and $l_{\pm} = 0$ according to the statement.
\end{proof}
\end{prop}

By Proposition \ref{p:complete}, if a complete rank 1 GKRS is symplectic-type and admits a further Killing symmetry which defines a local toric structure, then it is an $A$-deformation of a complete locally toric K\"ahler Ricci-flat metric. We detail this correspondence in the case $M=M_k^4$ is the toric resolution of $\C^2/\Z_{k+1}$:
\begin{prop} \label{p:toric-rank1} Suppose $(g_K, J_K)$ is  a multi-Euguchi-Hanson or a multi-Taub-NUT Ricci-flat toric K\"ahler metric on  $M^4_k$. Then the $A$-deformation of $(g_K, J_K)$ is a complete regular rank 1 GKRS described by item (1) or (2) of Proposition \ref{p:toricrank1class}.
\end{prop}

\begin{proof} By Lemma~\ref{toric-soliton}, the soliton vector fields of the $A$-deformation of the Ricci-flat K\"ahler structure on $(M^4_k, F, \T)$ are $X_J= \frac{1}{2} A(c)$ and $X_I=-\frac{1}{2}A(c)$. Notice that $(M^4, F, \T)$ is Delzant type, with a polytope $\Pol_k \subset (\R^2)^*$ obtained from the simplex $P_0=\{(x_1, x_2) \, |  \, L_1(x)=x_1 \geq 0, \, L_2(x)=x_2 \geq 0\}$  by intersection with $k$ additional hyperplanes $L_{3}(x) \geq 0, \ldots L_{k+2}(x) \geq 0$,  satisfying  $L_j(2, 2)=2$ (and of course the Delzant condition). This follows by Proposition~\ref{p:polytopecond}, which tells us that $c=(2,2)$ in the standard coordinates determined by the standard bases of the lattice $\Z^2$ circle subgroups of $\T$. We thus conclude $X_I= a(-1, 1)$ and $X_J=a(1, -1)$ in these coordinates, where $a\in \R$ parametrizes the bivector $A\in \Lambda^2 \tor \cong \R$.  This shows that $X_J$ generates a $S^1$-action, as required in \cite{SU2}.  A direct computation shows that the Ricci curvature is bounded below (cf. also Remark \ref{r:horizontalRicci}), and by construction $M$ has finite topology.  It follows that the result is described by Proposition \ref{p:toricrank1class}, and since by construction the result is symplectic-type it is described by item (1) or (2).
\end{proof}

\section{Four-dimensional rank \texorpdfstring{$2$}{2} GKRS} \label{s:reduction}

In this section we address the case of four-dimensional rank $2$ GKRS.  We begin by recalling the GK Gibbons Hawking ansatz of \cite{SU2}.  In the case the soliton is rank $2$, through a detailed local study we are lead to Proposition \ref{p:spanprop} which establishes the key property $\spn_{\mathbb R} \{I X_I, I X_J \} = \spn_{\mathbb R} \{J X_I, J X_J \}$ of \cite{boulanger2019toric}.  Together with Proposition \ref{p:4dlagrangian}, this implies that the structure is an $A$-deformation of a K\"ahler metric, and the theory of \S \ref{s:toricGKRS} and \S \ref{s:global} applies.  However we will proceed to give a direct derivation of the underlying symplectic potential and associated Monge-Amp\`ere equation to further illustrate its special properties in this setting.

\subsection{GK Gibbons Hawking Ansatz}\label{sec: GHansatz}

We fix a nondegenerate rank $2$ GKRS $(M^{4}, g, I, J)$, with vector fields $X_I, X_J$ preserving all structure.  In this subsection we recall the fundamental construction behind \cite{SU2}, which describes the local geometry of a four-dimensional generalized K\"ahler manifold with a biholomorphic Killing field.  Here we have two manifestations of this coming from $X_I$ and $X_J$.  First note that we have local $\R^3$-valued momentum maps,  $\nu^I$ and $\nu^J$,   with respect to the symplectic $2$-form $\Omega=\sigma^{-1}$ defined over the open dense subset where $\sigma \neq 0$, defined using the closed $1$-forms
\begin{align*}
i_{X_I}\Omega=d\nu^I_1,\quad i_{IX_I}\Omega=d\nu^I_2,\quad i_{JX_I}\Omega=d\nu^I_3,\\
i_{X_J}\Omega=d\nu^J_1,\quad i_{IX_J}\Omega=d\nu^J_2,\quad i_{JX_J}\Omega=d\nu^J_3.
\end{align*}
We will furthermore define the functions
\begin{align*}
    \nu^I_+ = \tfrac{1}{2} (\nu^I_2 + \nu^I_3), \quad \nu^I_- = \tfrac{1}{2} (\nu^I_2 - \nu^I_3), \quad \nu^J_+ = \tfrac{1}{2} (\nu^J_2 + \nu^J_3), \quad \nu^J_- = \tfrac{1}{2} (\nu^J_2 - \nu^J_3).
\end{align*}
Observe that $\nu^I_{-}, \nu^J_-$ are momenta for $F_+$, while $\nu^I_+, \nu^J_+$ are momenta for $F_-$.  Furthermore, define smooth functions
\begin{align*}
W_I = g(X_I,X_I)^{-1}, \qquad W_J = g(X_J,X_J)^{-1},
\end{align*}
and connection one-forms
\begin{align*}
\eta_I(Z) := W_I g(X_I, Z), \qquad \eta_J(Z) := W_J g(X_J, Z).
\end{align*}
These connections define splittings \begin{align*}
    TM = \mathbb R X_I \oplus \mathcal H_I = \mathbb RX_J \oplus \mathcal H_J.
\end{align*}
There exist canonical frames for $\mathcal H_I$ and $\mathcal H_J$, in particular
\begin{align*}
    \mathcal H_I = \IP{IX_I, JX_I, KX_I}, \quad \mathcal H_J = \IP{ IX_J, JX_J, KX_J}, \quad 
    K := \tfrac{1}{2} [I, J].
\end{align*}
The metric can be recovered from these data and is expressed using the angle function $p$:
\begin{align*}
p = - \tfrac{1}{4} \tr (IJ),
\end{align*}
for which we recall the key useful four-dimensional fact
\begin{align} \label{f:IJcomm}
\{I, J\} = - 2 p \mathbf{1}.
\end{align}
In fact, using these objects, we obtain the following descriptions of the metric and symplectic form $\Omega$.  First, the metric is
\begin{gather} \label{f:metricformula}
\begin{split}
    g =&\ W_I h_I + W_I^{-1} \eta_I^2, \quad h_I = (1-p^2) d (\nu^I_1)^2 + 2 (1 - p) (d\nu^I_+)^2 + 2 (1 + p) (d \nu^I_-)^2,\\
    g =&\ W_J h_J + W_J^{-1} \eta_J^2, \quad h_J = (1-p^2) (d \nu^J_1)^2 + 2 (1 - p) (d \nu^J_+)^2 + 2 (1 + p) (d\nu^J_-)^2,
\end{split}
\end{gather}
where $\eta_I$, $\eta_J$ are $1$-forms such that 
\begin{align*}
    d \eta_I =&\ \left( (W_I)_{\nu^I_3} + (p W_I)_{\nu^I_2} \right) d \nu^I_1 \wedge d \nu^I_2 - ((W_I)_{\nu^I_2} + (p W_I)_{\nu^I_3}) d \nu^I_1 \wedge d \nu^I_3 + (W_I)_{\nu^I_1} d \nu^I_2 \wedge d \nu^I_3,\\
    d \eta_J =&\ \left( (W_J)_{\nu^J_3} + (p W_Y)_{\nu^J_2} \right) d \nu^J_1 \wedge d \nu^J_2 - ((W_J)_{\nu^J_2} + (p W_J)_{\nu^J_3}) d \nu^J_1 \wedge d \nu^J_3 + (W_J)_{\nu^J_1} d \nu^J_2 \wedge d \nu^J_3,
\end{align*}
whereas $\Omega$ can be expressed as
\begin{gather} \label{f:OmegaGH}
\begin{split}
    \Omega =&\ - (d \nu^I_1 \wedge \eta_I + W_I d \nu^I_2 \wedge d \nu^I_3) = - (d \nu^J_1 \wedge \eta_J + W_J d \nu^J_2 \wedge d \nu^J_3)\\
    =&\ - d \nu^I_1 \wedge \eta_I + 2 W_I d \nu^I_+ \wedge d \nu^I_- = - d \nu^J_1 \wedge \eta_J + 2 W_J d \nu^J_+ \wedge d \nu^J_-.
\end{split}
\end{gather} 
Furthermore, as a consequence of the integrability condition for $d\eta_I$ and $d\eta_J$,  we obtain an a priori elliptic PDE that $W_X$ and $W_Y$ must satisfy:
\begin{gather}\label{eq: W}
\begin{split}
    0 =&\ (W_I)_{\nu^I_1 \nu^I_1} + \tfrac{1}{2} \left( ((1 + p) W_I)_{\nu^I_+ \nu^I_+} + ((1 - p) W_I)_{\nu^I_- \nu^I_-} \right),\\
    0 =&\ (W_J)_{\nu^J_1 \nu^J_1} + \tfrac{1}{2} \left( ((1 + p) W_J)_{\nu^J_+ \nu^J_+} + ((1 - p) W_J)_{\nu^J_- \nu^J_-} \right).
\end{split}
\end{gather}
The above provides a local description of all nondegenerate GKRS's, reminiscent of the Gibbons-Hawking description of hyperK\"ahler $4$-manifolds with a tri-holomorphic Killing vector field.

\subsection{Rank \texorpdfstring{$2$}{2} solitons as \texorpdfstring{$A$}{A}-deformations}

In this subsection we refine the construction of the previous subsection to give a more precise description of the geometry in terms of the rank $2$ action.  The key point is to establish the property $I \tor = J \tor$ of \cite{boulanger2019toric}, which then leads to the proof of Theorem \ref{t:4dmainthm}.  To begin we give a description of the Ricci potential $\Phi$ in terms of the momenta $\nu^I$ and $\nu^J$.

\begin{prop}\label{p:expressions with Hamiltonians} (cf. \cite[Proposition 4.8]{SU2}) Let $(M^{4}, g, I, J)$ be a nondegenerate GKRS.  Then we have 
\begin{align*}
    \Phi =&\ 2 \nu^J_3 - 2 \nu^I_2 + \mbox{const} = - 2 (\nu^I_+ + \nu^I_- - \nu^J_+ + \nu^J_-) + const, \\
    \N f =&\ \tfrac{1}{2} W_I^{-1} (\del_{\nu^I_+} - \del_{\nu^I_-}) - \tfrac{1}{2} W_J^{-1} (\del_{\nu^J_+} + \del_{\nu^J_-}),\\
    df =&\ (1 - p) d \nu^I_+ - (1 + p) d \nu^I_- - (1 - p) d \nu^J_+ - (1 + p) d \nu^J_-.
\end{align*}
\begin{proof} As the given GKRS is nondegenerate, we let $\Omega = \gs^{-1}$.  By Proposition \ref{p:GKRSsymmetries} we know that $V_I$ and $V_J$ preserve $\Omega$, with local Hamiltonian potentials $- \nu_2^I, - \nu_3^J$ respectively.  On the other hand, by Proposition \ref{p:transgression} we know that $2(V_I - V_J)$ is $\gs$-Hamiltonian with potential function $\Phi$, so the equation for $\Phi$ follows. 

Turning to $f$, we first express
\begin{align*}
    \N f = - V_I - V_J,
\end{align*}
so that
\begin{align*}
    i_{\N f} \Omega = - i_{V_I} \Omega - i_{V_J} \Omega = - d (\chi_{V_I} + \chi_{V_J}) = d( \nu^I_2 + \nu^J_3) = d (\nu^I_+ + \nu^I_- + \nu^J_+ - \nu^J_-).
\end{align*}
We now use the expressions (\ref{f:OmegaGH}) for $\Omega$ in momentum  coordinates
\begin{align*}
    \Omega = - d \nu^I_1 \wedge \eta_I + 2 W_I d \nu^I_+ \wedge d \nu^I_- = - d \nu^J_1 \wedge \eta_J + 2 W_J d \nu^J_+ \wedge d \nu^J_-,
\end{align*}
to obtain
\begin{align*}
    \N f = \tfrac{1}{2} W_I^{-1} (\del_{\nu^I_+} - \del_{\nu^I_-}) - \tfrac{1}{2} W_J^{-1} (\del_{\nu^J_+} + \del_{\nu^J_-}),
\end{align*}
Finally, using the explicit form (\ref{f:metricformula}) of the metric we arrive at
\begin{align*}
    d f = (1 - p) d \nu^I_+ - (1 + p) d \nu^I_- - (1 - p) d \nu^J_+ - (1 + p) d \nu^J_-,
\end{align*}
as claimed.
\end{proof}
\end{prop}


Using this understanding of $\Phi$ we can give a description of $X_J$ in terms of the tangent space decomposition induced by $X_I$, and vice versa.

\begin{lemma}\label{lem: Y_decomposition}
    Let $(M^{4}, g, I, J)$ be a nondegenerate GKRS. Then, in the locus where $X_I,X_J$ are not zero, one has \begin{equation}\label{eq:Y_decomposition}
        X_J = \eta_I(X_J) X_I + \gl W_I KX_I, \quad X_I = \eta_J(X_I) X_J - \gl W_J KX_J,
    \end{equation}
    where $\lambda \equiv \Omega(X_I,X_J)$ is constant.
\end{lemma}
\begin{proof}
    First of all, the orthogonal decomposition provided by $\eta_I$ gives
    \begin{align*}
        X_J = \eta_I(X_J) X_I + X_J^H.
    \end{align*}
    We will deduce that $X_J^H$ is a multiple of $KX_I$ using the equation for $\Phi$ in Proposition \ref{p:expressions with Hamiltonians}.  In any chart there exist Hamiltonian potentials for $V_I=-IX_I$ and $V_J=-JX_J$, so we have the local expression
    \[
    \Phi=2\nu^J_3-2\nu^I_2+const.
    \]
    Now, $i_{X_I}d\Phi=L_{X_I}\Phi=0$ and similarly $i_{X_J} d \Phi = L_{X_J}\Phi=0$. Furthermore, using that $\Omega$ is of type $(2,0) + (0,2)$ with respect to both $I$ and $J$, we have that $i_{X_I}d\nu^I_2= i_{X_J} d \nu^J_3 = 0$. 
    Consequently,
    \begin{align*}
    i_{X_I}d\nu^J_3= i_{X_I} d (\Phi + 2 \nu^I_2) = \Omega(JX_J,X_I)=0,\\
    i_{X_J}d\nu^I_2=i_{X_J} d (2 \nu^J_3 - \Phi) = \Omega(IX_I,X_J)=0.
    \end{align*}
    Using again that $\Omega$ is type $(2,0) + (0,2)$, it furthermore follows that
    \begin{equation*}
    i_{X_J}d\nu^I_3=0,\qquad i_{X_I}d\nu^J_2=0.
    \end{equation*}
    It follows that
    \begin{align*}
        X_J^H \nu^I_2 = i_{X_J^H} d \nu^I_2 = i_{X_J} d \nu^I_2 - i_{\eta_I(X_J) X_I} d \nu^I_2 = 0,\\
        X_J^H \nu^I_3 = i_{X_J^H} d \nu^I_3 = i_{X_J} d \nu^I_3 - i_{\eta_I(X_J) X_I} d \nu^I_3 = 0.
    \end{align*}
    It follows that there exists a function $\gl$ such that $X_J^H = \gl \frac{\del}{\del \nu^I_1}$.  Note that it follows from the definitions of $\nu^I$ that $\frac{\del}{\del \nu^I_1} = - W_I K X_I$, hence $X_J^H = - \gl W_I KX_I$, as claimed.  
    
    It remains to show that $\gl$ is actually constant.  First notice that since $[X_I, X_J] = 0$ and $X_J$ preserves the GK structure it follows that $L_{X_J} d \nu^I_1 = L_{X_J} (i_{X_I} \Omega) = 0$.  Since also $i_{X_I} d \nu^I_1 = 0$ it follows that
    $$
    L_{X_J^H} d \nu^I_1 = L_{X_J}d\nu^I_1 - L_{\eta_I(X_J)X_I}d\nu^I_1 = - d i_{\eta_I(X_J) X_I} d \nu^I_1 = 0,
    $$ 
    giving the claim that $\gl$ is constant.  Finally we compute $\gl$ by observing
    \[
    \Omega(X_I,X_J)=\gl W_I\Omega(X_I,KX_I)=\gl.
    \]
    The same argument can be applied to decompose $X_I$ as claimed.
\end{proof}

\begin{rmk}
    The case $\gl = 0$ corresponds to a rank $1$ soliton while for rank $2$ solitons $\gl\neq0$.  Notice also this this implies that the rank cannot change in the nondegeneracy locus of $M$.
\end{rmk}

Using the description of $X_I$ and $X_J$ above, we next establish the claim that $I \tor_p = J \tor_p$.  This is precisely the condition required by \cite[Prop. 4]{boulanger2019toric} for a $\tor$-invariant GK structure $(g, I, J, b)$ tamed by a symplectic form to be obtained by the construction of \S \ref{ss:Atransform}.  While we could then immediately plug in to the general theory developed above, it is relevant for the examples to follow to explicitly derive further geometric characteristics of this particular setting, and to give an a priori derivation of the underlying symplectic potential and the relevant Monge-Amp\`ere equation.

\begin{lemma} \label{p:spanprop}
	Let $(M^4,g,I,J)$ be a nondegenerate rank 2 GKRS, and let $\tor = \spn_{\mathbb R} \{X_I, X_J \}$.  At points where $\dim \tor = 2$ we have
	\[
	I \tor = J \tor.
	\]
\end{lemma}
\begin{proof}
	Using the expression for $X_J$ from Lemma \ref{lem: Y_decomposition} and (\ref{f:IJcomm}) we compute
    \begin{align*}
        IX_J =&\ \eta_I(X_J) IX_I + \gl W_I IK X_I = \eta_I(X_J) IX_I + \gl W_I (- J + 2p I ) X_I.
    \end{align*}
    Since $\gl \neq 0$ and $W_I \neq 0$, this gives $JX_I \in \spn_{\mathbb R} \{ IX_I, IX_J\}$.  Similarly we obtain $JX_J \in \spn_{\mathbb R} \{ IX_I, IX_J \}$, and the proposition follows.
\end{proof}

\begin{proof}[\bf Proof of Theorem \ref{t:4dmainthm}] Given $(M^4, g, I, J)$ a rank $2$ GKRS with $\gs$ not vanishing identically, we know by \cite{AGG} (cf. also \cite{ASU3} \S 2) that it is symplectic-type on a dense open set.  This statement holds relative to both $F_{\pm}$, and we choose $F = F_+$.  The rank $2$ hypothesis gives a $2$-dimensional space $\tor = \spn_{\mathbb R} \{X_I, X_J\}$ of Killing vector fields preserving the structure.  We know by Proposition \ref{p:4dlagrangian} that $F_{|\tor \times \tor} \equiv 0$, and by Lemma \ref{p:spanprop} that $I \tor = J \tor$.  Thus by Theorem \ref{t:Atypeconstruction} $(g, I, J)$ is an $A$-deformation of a locally toric KRS as claimed.  The completeness of this toric KRS in the case the GKRS is symplectic-type follows from Proposition \ref{p:complete}, whence it follows from Corollary \ref{c:rank1fromCY} that the toric KRS is not Ricci-flat.
\end{proof}

\subsection{Monge Amp\`ere reduction}

Here we give an alternative derivation of the Monge-Amp\`ere equation underlying a four-dimensional rank $2$ GKRS.  To begin we observe identities relating the different momenta:

\begin{lemma} Let $(M^{4}, g, I, J)$ be a nondegenerate rank $2$ GKRS.  Then
    \begin{gather}\label{eq:A}
    \begin{split}
    d\nu^I_-=&\ \frac{1}{\gl(1+p)}(|X_I|^2d\nu^J_+-g(X_I,X_J)d\nu^I_+),\\
    d\nu^J_-=&\ \frac{1}{\gl(1+p)}(g(X_I,X_J)d\nu^J_+-|X_J|^2d\nu^I_+),\\
    d\nu^I_+=&\ \frac{1}{\gl(1-p)} \left( g(X_I,X_J) d \nu^I_- -\brs{X_I}^2 d \nu^J_- \right),\\
    d\nu^J_+=&\ \frac{1}{\gl(1-p)} \left( \brs{X_J}^2 d \nu^I_- - g(X_I,X_J) d \nu^J_- \right).
    \end{split}
    \end{gather}

\begin{proof}
    Let's start by noticing that the Hamiltonian function $\nu^J_3$ satisfies
    \begin{align*}
    d\nu^J_3=&\ i_{JX_J}\Omega\\
    =&\ i_{ \eta_I(X_J) JX_I + \gl W_I JK X_I} \Omega\\
    =&\ \eta_I(X_J)  d \nu^I_3 + \gl W_I i_{(I - p J)X_I} \Omega\\
    =&\ \eta_I(X_J) d\nu^I_3+\gl W_I(d\nu^I_2-pd\nu^I_3)\\
    =&\ \frac{g(X_I,X_J)-\gl p}{|X_I|^2}d\nu^I_3+\frac{\gl}{|X_I|^2}d\nu^I_2,
    \end{align*}
    and performing the same calculation for $\nu^J_2$ and interchanging the roles of $\nu^I$ and $\nu^J$ we arrive at equations
    \begin{align*}
    |X_I|^2d\nu^J_3&=(g(X_I,X_J)-\gl p)d\nu^I_3+\gl d\nu^I_2,\\
    |X_I|^2d\nu^J_2&=(g(X_I,X_J)+\gl p)d\nu^I_2-\gl d\nu^I_3,\\
    |X_J|^2d\nu^I_3&=(g(X_I,X_J)+\gl p)d\nu^J_3-\gl d\nu^J_2,\\
    |X_J|^2d\nu^I_2&=(g(X_I,X_J)-\gl p)d\nu^J_2+\gl d\nu^J_3.
    \end{align*}
    The lemma is an elementary consequence of these identities.
\end{proof}
\end{lemma}

\begin{lemma} \label{l:potentialderivation} Let $(\mathring{M}^{4}, g, I, J)$ be a nondegenerate rank $2$ GKRS.  There exist locally defined $\tor$-invariant potential functions  $\varphi, \psi$ such that
    \begin{gather*}
        \gl(1+p) \left(\begin{matrix}
    \varphi_{\nu^I_+\nu^I_+} & \varphi_{\nu^I_+\nu^J_+}\\ \varphi_{\nu^J_+\nu^I_+} & \varphi_{\nu^J_+\nu^J_+}
    \end{matrix}\right)=
    \left(\begin{matrix}
    |Y|^2 & -g(X,Y)\\ -g(X,Y) & |X|^2
    \end{matrix}\right) = \gl(1 - p) 
    \begin{pmatrix}
        \psi_{\nu^I_-\nu^I_-} & \psi_{\nu^I_-\nu^J_-} \\
        \psi_{\nu^J_-\nu^I_-} & \psi_{\nu^J_-\nu^J_-}
    \end{pmatrix}.
    \end{gather*}
\end{lemma}
\begin{proof}
    It follows from (\ref{eq:A}) that the one forms
    \begin{align*}
         -\nu^J_- d \nu^I_+ + \nu^I_- d \nu^J_+, \qquad \nu^J_+ d \nu^I_- - \nu^I_+ d \nu^J_-,
    \end{align*}
    are both closed, and thus locally there exist potential functions $\varphi$ and $\psi$ such that
    \begin{align} \label{f:Legendre}
    \frac{\del \varphi}{\del \nu^I_+}= -\nu^J_-, \qquad \frac{\del \varphi}{\del \nu^J_+}= \nu^I_-, \qquad \frac{\del \psi}{\del \nu^I_-} = \nu^J_+, \qquad \frac{\del \psi}{\del \nu^J_-} = - \nu^I_+.
    \end{align}
    Returning to equations \eqref{eq:A} we easily derive the claimed expressions.
\end{proof}

\begin{prop} \label{p:MAreduction} 
    Let $(M^{4}, g, I, J)$ be a nondegenerate rank $2$ GKRS.  The potential functions $\varphi, \psi$ of Lemma \ref{l:potentialderivation} solve
    \begin{align} \label{f:MAreduction}
    \det(\mathrm{Hess}(\varphi)) = e^{2\left(\nu^J_+-\nu^I_+ -\varphi_{\nu^J_+} + \varphi_{\nu^I_+}\right)}, \quad 
    \det \Hess(\psi) = e^{2\left(\nu^I_- + \nu^J_- - \psi_{\nu^I_-} - \psi_{\nu^J_-}\right)}.
    \end{align}
\end{prop}
\begin{proof} 
    We first recall that in the Gibbons-Hawking ansatz we have $\brs{KX_I}^2 = (1 - p^2)\brs{X_I}^2$.
    Using this and the expression for $\brs{X_J}^2$ coming from~\eqref{eq:Y_decomposition} we get the identity
    \begin{gather*}
    \begin{split}
    |X_I|^2|X_J|^2-g(X_I,X_J)^2=&\ \brs{X_I}^2 \left( \eta_I(X_J)^2 \brs{X_I}^2 + \gl^2 \brs{X_I}^{-4} \brs{KX_I}^2 \right) - g(X_I,X_J)^2\\
    =&\ \brs{X_I}^4 \left( W_I g(X_I,X_J) \right)^2 + \gl^2 (1 - p^2) - g(X_I,X_J)^2\\
    =&\ \gl^2(1-p^2).
    \end{split}
    \end{gather*}
    Applying it to the result of Lemma \ref{l:potentialderivation}, we obtain the Monge-Amp\`ere type equation
    \[
    \det(\mathrm{Hess}(\varphi))= \frac{1-p}{1+p} = e^{\Phi} = \frac{1}{\det \mathrm{Hess}(\psi)}.
    \]
    Using the prescribed expression for $\Phi$ from Proposition \ref{p:expressions with Hamiltonians}, and equations (\ref{f:Legendre}) we obtain
    \[
    \Phi=2(\nu^J_+-\nu^J_--\nu^I_+-\nu^I_-)=2(\nu^J_+ - \nu^I_+ - \varphi_{\nu^J_+} + \varphi_{\nu^I_+}) = 2(\psi_{\nu^I_-} + \psi_{\nu^J_-} - \nu^I_- - \nu^J_-).
    \]
    Thus we obtain
    \[
    \det(\mathrm{Hess}(\varphi)) = e^{2\left(\nu^J_+-\nu^I_+ -\varphi_{\nu^J_+} + \varphi_{\nu^I_+}\right)}, \quad 
    \det \Hess(\psi) = e^{2\left(\nu^I_- + \nu^J_- - \psi_{\nu^I_-} - \psi_{\nu^J_-}\right)},
    \]
    as claimed.
\end{proof}
  
We now deduce an explicit formula for the soliton potential $f$ in the coordinates $(\nu^I_-,\nu^J_-)$ and $(\nu^I_+,\nu^J_+)$.
\begin{prop} \label{p:explicitf} 
    There exists a constant $c \in \mathbb R$ so that the function $\til{f} = f - c$ satisfies
    \begin{align*}
        \til{f}=&\ \log(1-p) - 2\psi_{\nu^I_-} - 2\psi_{\nu^J_-} = \log(1+p) - 2\varphi_{\nu^J_+} + 2\varphi_{\nu^I_+} .
    \end{align*}
\end{prop}
\begin{proof}
    We compute, starting from the formula for $df$ in Proposition \ref{p:expressions with Hamiltonians} and using (\ref{f:Legendre}),
    \begin{align*}
        df =&\ (1-p)(d\nu^I_+-d\nu^J_+) - (1+p)(d\nu^I_-+d\nu^J_-)\\
        =&\ (1-p)\left[ - d\left(\psi_{\nu^J_-} + \psi_{\nu^I_-}\right) - \frac{1+p}{1-p}\left(d\nu^I_-+d\nu^J_-\right) \right]\\
        =&\ \frac{1-p}{2}\left[ - d\left(2\psi_{\nu^J_-} + 2\psi_{\nu^I_-}\right) - e^{2\left(- \psi_{\nu^I_-} - \psi_{\nu^J_-} + \nu^I_- + \nu^J_-\right)}(2d\nu^I_- + 2d\nu^J_-) \right]\\
        =&\ \frac{1-p}{2}\left[ - d\left(2\psi_{\nu^J_-} + 2\psi_{\nu^I_-}\right) - d\left(e^{2\left(- \psi_{\nu^I_-} - \psi_{\nu^J_-} + \nu^I_- + \nu^J_-\right)}\right) - e^{2\left(- \psi_{\nu^I_-} - \psi_{\nu^J_-} + \nu^I_- + \nu^J_-\right)}d(2\psi_{\nu^J_-} + 2\psi_{\nu^I_-}) \right]\\
        =&\ \frac{1-p}{2}\left[ - d\left(e^{2\left(- \psi_{\nu^I_-} - \psi_{\nu^J_-} + \nu^I_- + \nu^J_-\right)}\right) - \left(1+e^{2\left(-\psi_{\nu^I_-} - \psi_{\nu^J_-} + \nu^I_- + \nu^J_-\right)}\right)d\left(2\psi_{\nu^J_-} + 2\psi_{\nu^I_-}\right) \right]\\
        =&\ - \frac{1-p}{2} d\left(\frac{1+p}{1-p}\right) - 2d\left(\psi_{\nu^J_-} + \psi_{\nu^I_-}\right)\\
        =&\ -\frac{dp}{1-p} - 2d\left(\psi_{\nu^J_-} + \psi_{\nu^I_-}\right)\\
        =&\ d\left( \log(1-p) - 2\psi_{\nu^J_-} - 2\psi_{\nu^I_-}\right),
    \end{align*}
    thus it follows that there exists a constant $c$ so that
    \begin{align*}
    f = \log (1 - p) - 2\psi_{\nu^J_-} - 2\psi_{\nu^I_-} + c.
    \end{align*}
    By a similar computation in the $(\nu^I_+,\nu^J_+)$ coordinates we obtain that there exists a constant $c'$ such that
    \begin{align*}
    f = \log (1 + p) - 2\varphi_{\nu^J_+} + 2\varphi_{\nu^I_+} + c'.
    \end{align*}
    Subtracting these equations we find using $\tfrac{1-p}{1+p}=e^\Phi$ that the two constants are the same, and the proposition follows.
\end{proof}

\begin{rmk}
    Recall that $2\nu_\pm^I,2\nu_\pm^J$ are momenta for $F_\mp$ with respect to $X_I,X_J$, which were called $\mu_\mp$ in \S \ref{s:toricGKRS}.
    Setting $F=F_+$, by Theorem \ref{t:4dmainthm} $(M^4,g,I,J,\t)$ is an $A$-deformation of a locally toric KRS $(M^4,g_u,J_u,F,\t)$. It can be checked that $u=\tfrac{1}{\gl} \psi$ and $A=\begin{pmatrix}
        0 & \tfrac{1}{\gl}\\
        -\tfrac{1}{\gl} & 0
    \end{pmatrix}$.
    The MA equation \eqref{f:MAreduction} then reads as
    $$\log\det\Hess(u) + \log\gl^2 = -\gl\left<\N u , X_I^*+X_J^*\right> + \left<\mu_+ , X_I+X_J\right>,$$
    hence matching \S \ref{s:toricGKRS} with $b=(X_I+X_J)$ and $c=A^{-1}(X_J-X_I)$. Similarly, for the soliton potential $\tilde f$ computed in Proposition \ref{p:explicitf} we notice that $1-p = \tfrac{2}{\gl^2\det(\Hess\ u\ +A)}$ hence it holds
    \begin{align*}
        \tilde f &= -\log\det(\Hess(u) + A) - \la\nabla u,c\ra +\log(2\gl^2),
    \end{align*}
    matching the equation in Remark \ref{r:GKRSpotential}.
\end{rmk}


\end{document}